\documentclass[A4paper]{article}

\usepackage{lipsum}        

\usepackage[a4paper,top=3cm,bottom=2cm,left=3cm,right=3cm,marginparwidth=1.75cm]{geometry}
\usepackage[utf8]{inputenc}
\usepackage{authblk}
\usepackage{amsmath}
\usepackage{mathtools}
\usepackage{cases}
\usepackage{bm}
\usepackage{amsfonts}
\usepackage{subfiles}
\usepackage{multicol}
\usepackage{tabularx,enumerate}
\usepackage{subcaption}
\usepackage{tikz}
\usepackage{caption}
\usepackage[absolute,overlay]{textpos}
\usepackage{ifthen}
\usepackage{tcolorbox}
\usepackage{mathrsfs}
\usepackage{scalerel}
\usepackage{graphics}

\usepackage{soul}%to draw a horizontal line on words

\usepackage{comment}

\usepackage{todonotes}

\tcbuselibrary{theorems}
\tcbuselibrary{skins}

\usepackage{algorithm}
\usepackage{algpseudocode}
\usepackage{mycommands}
\usepackage{xcolor}
\usepackage{amsthm,amssymb}
\newtheorem{theorem}{Theorem}[section]

\newtheorem{lemma}[theorem]{Lemma}

\newtheorem{remark}[theorem]{Remark}

\newtheorem{assum}{Assumption}

\theoremstyle{definition}

\usepackage[toc,page]{appendix} 
\usepackage{hyperref}
\hypersetup{
     colorlinks   = true,
     linkcolor    = blue,
     citecolor    = red
}

\title{Global non-asymptotic super-linear convergence rates of regularized proximal quasi-Newton methods on non-smooth composite problems}
\author[*]{Shida Wang}
\author[**]{Jalal Fadili}
\author[*]{Peter Ochs}
\affil[*]{Department of Mathematics and Computer Science, Saarland University, Germany}
\affil[**]{Normandie Universit\'e, ENSICAEN, UNICAEN, CNRS, GREYC, France.}
\date{\today}

\begin{document}
\maketitle
\begin{abstract}
    In this paper, we propose two regularized proximal quasi-Newton methods with symmetric rank-1 update of the metric (SR1 quasi-Newton) to solve non-smooth convex additive composite problems. Both algorithms avoid using line search or other trust region strategies. For each of them, we prove a super-linear convergence rate that is independent of the initialization of the algorithm. The cubic regularized method achieves a rate of order $\left(\frac{C}{N^{1/2}}\right)^{N/2}$, where $N$ is the number of iterations and $C$ is some constant, and the other gradient regularized method shows a rate of the order $\left(\frac{C}{N^{1/4}}\right)^{N/2}$. To the best of our knowledge, these are the first global non-asymptotic super-linear convergence rates for regularized quasi-Newton methods and regularized proximal quasi-Newton methods. The theoretical properties are also demonstrated in two applications from machine learning.
\end{abstract}
\section{Introduction}
Newton-type methods have been studied extensively over decades due to their fast convergence. However, the complexity of the computational cost of Newton's method per iteration is cubic, making it intractable for large-scale applications. Therefore, concurrently, quasi-Newton methods have been developed to avoid the explicit computation of the Hessian (second derivative) of the objective \cite{dennis1977quasi}. The idea of quasi-Newton methods is to approximate the Hessian by a matrix that is generated from first-order information, for example, using the difference of the gradient of the objective at nearby points. Based on this idea, numerous variants have been developed, including BFGS \cite{broyden1970convergence}, SR1 \cite{davidon1991variable,broyden1967quasi}, DFP \cite{davidon1991variable,fletcher1963rapidly}. Classically, the convergence guarantee of both, Newton's method and quasi-Newton methods, are local, i.e., they require the starting point to be sufficiently close to an optimal point, unless globalization strategies such as trust region or line search are applied \cite{conn2000trust}. This picture has changed thanks to the pioneering work of Nesterov and Polyak \cite{nesterov2006cubic}. They propose a cubic regularization strategy for Newton's method that guarantees global convergence without line search. More recently, following the same goal, gradient regularization was introduced to Newton's method \cite{mishchenko2023regularized}. While these regularization strategies stabilize the algorithm and allow for global convergence, they do not remedy the enormous computational cost of involving the Hessian of the objective in the update step. Therefore, it is natural to take this as inspiration for the design of regularized quasi-Newton methods that converge globally and are also applicable to large-scale problems. However, even for strongly convex functions, the question remains:\begin{center}
    \emph{how can we design a globally convergent cubic- (or gradient-) regularized quasi-Newton method following the spirit of the pioneering works \cite{nesterov2006cubic,mishchenko2023regularized}.}
\end{center}

A striking property of quasi-Newton-type methods is their super-linear rate of convergence  \cite{broyden1970convergence,broyden1973local,dixon1972quasi,fletcher1970new,goldfarb1970family}, which however in most cases can only be proved locally or are only asymptotic rates. 
 Recently, \cite{rodomanov2021greedy} provides the first non-asymptotic super-linear convergence rate for greedy quasi-Newton methods, which is refined to explicit rates for a restricted Broyden family of quasi-Newton methods in  \cite{rodomanov2022rates, jin2023non}. In \cite{ye2023towards}, the first non-asymptotic explicit super-linear convergence rate for the SR1 method is proved. While these rates are non-asymptotic, they are still local, meaning that the initialization is required to be sufficiently close to an optimal point. In this sense, a local region around an optimal point is defined in which the super-linear rate of convergence can be observed. This raises the natural question of
\emph{whether it is possible to design quasi-Newton methods with global convergence and global non-asymptotic super-linear convergence rates.}

In this paper, we give an affirmative answer to both questions above by proposing cubic- and gradient-regularized quasi-Newton methods. In Table~\ref{table:comparisonsmooth}, we list the main features and limitations of the state-of-the-art quasi-Newton methods and arrange our proposed regularized quasi-Newton methods in this context, which we explain in more detail in the following. Unlike other works mentioned in Table~\ref{table:comparisonsmooth} and Table~\ref{tab:comparison}, our methods attain a non-asymptotic super-linear rate of convergence that is independent of the initialization and  without using expensive line search or trust region sub-routines. The globalization is achieved via regularization and the regime of super-linear convergence is encoded in the number of iterations instead of a hard-to-estimate local region around an optimal point. Instead of assuming the Dennis--Mor\'e criterion for convergence \cite{chen1999proximal,lee2012proximal}, which is usually difficult to validate, we rely on Lipschitz continuity of the Hessian instead. 
\begin{table}[htp]
\begin{tabular}{|lll|}
\hline
\multicolumn{3}{|c|}{State-of-the-art quasi-Newton-type methods with super-linear rate of convergence} \\ \hline
\multicolumn{1}{|l||}{Scheme}                       & \multicolumn{1}{c|}{Rate} & \multicolumn{1}{c|}{Local region of super-linear convergence}  \\ \hline\hline
% \multicolumn{1}{|l|}{Greedy \cite{rodomanov2021greedy}}                       & \multicolumn{1}{l|}{ $\left(\frac{nL^2}{\mu^2 N}\right)^{N/2}$  }                 & \multicolumn{1}{l|}{ $\frac{L_H}{\mu^{3/2}}\lambda_f(x_0)\leq \frac{\log \frac{3}{2} \mu}{4L}$}                                       \\ \hline
\multicolumn{1}{|l||}{DFP \cite{rodomanov2022rates}}                       & \multicolumn{1}{c|}{ $\left(\frac{nL^2}{\mu^2 N}\right)^{N/2}$  }                 & \multicolumn{1}{c|}{ $\frac{L_H}{\mu^{3/2}}\lambda_f(x_0)\leq \frac{\log \frac{3}{2} \mu}{4L}$}                                 \\ \hline
\multicolumn{1}{|l||}{BFGS \cite{rodomanov2022rates}}                       & \multicolumn{1}{c|}{ $\left(\frac{nL}{\mu N}\right)^{N/2}$ }                 & \multicolumn{1}{c|}{ $\frac{L_H}{\mu^{3/2}}\lambda_f(x_0)\leq \frac{\log \frac{3}{2} \mu}{4L}$}                                        \\ \hline
\multicolumn{1}{|l||}{DFP \cite{jin2023non}}                         & \multicolumn{1}{c|}{$\left(\frac{1}{ N}\right)^{N/2}$  }                 & \multicolumn{1}{c|}{ $\frac{L_H}{\mu^{3/2}}\norm[]{\nabla^2 f(x_*)^{1/2}(x_0-x_*)}\leq \min\{\frac{1}{120},\frac{1}{7\sqrt{n}}\}$}                                  \\ \hline
\multicolumn{1}{|l||}{BFGS \cite{jin2023non}}                         & \multicolumn{1}{c|}{$\left(\frac{1}{ N}\right)^{N/2}$  }                 & \multicolumn{1}{c|}{ $\frac{L_H}{\mu^{3/2}}\norm[]{\nabla^2 f(x_*)^{1/2}(x_0-x_*)}\leq \min\{\frac{1}{50},\frac{1}{7\sqrt{n}}\}$}                             \\ \hline
\multicolumn{1}{|l||}{SR1 \cite{ye2023towards}}                          & \multicolumn{1}{c|}{$\left(\frac{2n\log4\varkappa}{N}\right)^{N/2}$}                 & \multicolumn{1}{c|}{ $\frac{L_H}{\mu^{3/2}}\lambda_f(x_0)\leq\frac{\log \frac{3}{2}}{4\sqrt{\frac{3}{2}}}\max\{\frac{1}{2\varkappa}, \frac{1}{3n\log4\varkappa+3}\}$}                                          \\ \hline
\multicolumn{1}{|l||}{Cubic SR1 PQN (Alg~\ref{Alg:SR1_da_PQN})  }    & \multicolumn{1}{c|}{ $\left(\frac{C}{N^{1/2}}\right)^{N/2} $}                 & \multicolumn{1}{c|}{global}                   \\ \hline
\multicolumn{1}{|l||}{Grad  SR1 PQN (Alg~\ref{Alg:SR1_da_PQN2})} & \multicolumn{1}{c|}{$\left(\frac{C_{\mathrm{grad}}}{N^{1/4}}\right)^{N/2} $}            & \multicolumn{1}{c|}{global}                                \\ \hline
\end{tabular}
\centering
\caption{We compare our methods with classical quasi-Newton methods on the smooth problem $\min_{x\in\R^n} f$ where $f$ has $L$-Lipschitz gradient and is $\mu$-strongly convex. Here, we rewrite the results in our notations. $N$ is the number of iterations, $x_*$ denotes the optimal point, $\lambda_f(x_0):=\norm[\nabla^2 f(x_0)^{-1}]{\nabla f(x_0)}$ and $\varkappa\coloneqq\frac{L}{\mu}$. $C$, $C_{\mathrm{grad}}$ are constants (see Theorems~\ref{thm:main-cubic} and \ref{thm:main-grad}). }
\label{table:comparisonsmooth}
\end{table}

\begin{table}[htp]

\begin{tabular}{|llllc|}
\hline
\multicolumn{5}{|c|}{Comparison of recent works on convergence rates of quasi-Newton methods}                                                                                                                       \\ \hline
\multicolumn{1}{|l||}{Results}        & \multicolumn{1}{l|}{Non-asymptotic} & \multicolumn{1}{l|}{Explicit} & \multicolumn{1}{l|}{Global} & \begin{tabular}[c]{@{}c@{}}Without Line search\\ and trust region strategy\end{tabular} \\ \hline\hline
\multicolumn{1}{|l||}{\begin{tabular}[c]{@{}c@{}}Broyden family\\ \cite{rodomanov2021greedy,rodomanov2022rates,jin2023non}\end{tabular}} & \multicolumn{1}{c|}{Yes}          & \multicolumn{1}{c|}{Yes}      & \multicolumn{1}{c|}{No}     & Yes                                                                         \\ \hline
\multicolumn{1}{|l||}{SR1 \cite{ye2023towards}}            & \multicolumn{1}{c|}{Yes}          & \multicolumn{1}{c|}{Yes}      & \multicolumn{1}{c|}{No}     & Yes                                                                         \\ \hline
\multicolumn{1}{|l||}{BFGS \cite{jin2024non}}           & \multicolumn{1}{c|}{Yes}          & \multicolumn{1}{c|}{Yes}      & \multicolumn{1}{c|}{Yes}    & No (exact line search)                                               \\ \hline
\multicolumn{1}{|l||}{BFGS \cite{jin2024nonAmijo,rodomanov2024global}}           & \multicolumn{1}{c|}{Yes}          & \multicolumn{1}{c|}{Yes}      & \multicolumn{1}{c|}{Yes}    & No (inexact line search)                                             \\ \hline
\multicolumn{1}{|l||}{BFGS \cite{jiang2023online}}           & \multicolumn{1}{c|}{Yes}          & \multicolumn{1}{c|}{Yes}      & \multicolumn{1}{c|}{Yes}    & \multicolumn{1}{c|}{Yes (online learning strategy)}                   \\ \hline
\multicolumn{1}{|l||}{Our Alg \ref{Alg:SR1_da_PQN} and~\ref{Alg:SR1_da_PQN2}}           & \multicolumn{1}{c|}{Yes}          & \multicolumn{1}{c|}{Yes}      & \multicolumn{1}{c|}{Yes}    & \multicolumn{1}{c|}{Yes (cubic- or grad-regularization)}                  \\ \hline
\end{tabular}
\centering
\caption{Works related to ours that have achieved local or global convergence with asymptotic or non-asymptotic explicit  super-linear rate of convergence.}
\label{tab:comparison}
\end{table}

The analysis of our algorithms is built on a novel Lyapunov-type potential function that is in fact simpler than those employed in related work. We summarize recently used potential functions in Table~\ref{table:PotentialFunction}. While our algorithms and analysis already demonstrate a significant contribution for smooth optimization problems, furthermore, we consider non-smooth additive composite problems for which proximal gradient type algorithms are the standard choice for the large-scale regime. Such problems frequently appear in several areas in machine learning, computer vision, image or signal processing, and statistics. In this structured non-smooth case, the quasi-Newton metric is constructed using the smooth part of the objective function, whereas the non-smooth part of the objective is handled via a proximal step that is computed in the same metric. 

\begin{table}[htp]
\begin{tabular}{|ll|}
\hline
\multicolumn{2}{|c|}{Potential function}                                                                 \\ \hline\hline
\multicolumn{1}{|l|}{Broyden Family \cite{rodomanov2022rates}}     & $\sigma(A,G)=\trace(A^{-1}(G-A))$ for $A\preceq G$            \\ \hline
\multicolumn{1}{|l|}{Broyden Family \cite{rodomanov2022rates,byrd1987global}}     & $   \sigma(A,G)= \trace(A^{-1}(G-A))-\log \det(A^{-1}G)$ for $A\preceq G$ \\ \hline
\multicolumn{1}{|l|}{Broyden Family \cite{jin2023non,broyden1973local}}     & Frobenius-norm involved potential functions                            \\ \hline
\multicolumn{1}{|l|}{SR1\cite{ye2023towards}, Broyden Family \cite{rodomanov2021new}   }             & $\sigma(A,G)=\log\det (A^{-1}G)$  and $\trace(G-A)$   for $\sigma(A,G)=A\preceq G $                    \\ \hline
\multicolumn{1}{|l|}{Greedy SR1\cite{lin2021greedy}}                &  $\sigma(A,G)=\trace(G-A)$   for $A\preceq G $                    \\ \hline
\multicolumn{1}{|l|}{Our Cubic/Grad SR1 PQN} & $V(G) = \trace G$                                            \\ \hline
\end{tabular}
\centering
\caption{We list the potential functions introduced or used in the literature to prove super-linear convergence rates of quasi-Newton methods along with our potential function. Here $\trace G$ denotes the trace of the matrix $G$ and $\det G$ denotes the determinant of the matrix $G$. }
\label{table:PotentialFunction}
\end{table}
% We think it is fair to assume Lipschitz condition on Hessian as many previous work in Newton type method\cite{nesterov2006cubic,mishchenko2023regularized}. In \cite{ye2023towards,rodomanov2021greedy}, they consider strongly self-concordant function, which is slightly more general from Lipschitz condition on Hessian. Our results in this paper still hold if we replace the Lipschitz on Hessian by strongly self-concordant function. Since this paper is a first try in studying the convergence of regularized quasi-Newton method without line search or trust region, we think strong convexity can be also tolerated.

%%\textbf{Summary of main contributions}
%%\begin{enumerate}
%%    \item We propose a new potential function $V(G)$ to analyze the convergence of quasi-Newton metric.
%%    \item Our cubic SR1 (proximal) quasi-Newton method has a global convergence with explicit (non-asymptotic) super-linear rate.
%%    \item Based on our cubic SR1 proximal quasi-Newton method, we propose two gradient regularized methods which use variants of gradient regularization and a restarting strategy. Both methods have global convergences with explicit (non-asymptotic) super-linear rates. Additionally, if applied to a smooth problem, the complexity of the update step from one of our gradient regularized methods is $O(n^2)$ by using Sherman-Morrison-Woodbury formula. 
%%   
%%\end{enumerate}
%%

\section{Related Works}

\paragraph{Regularized Newton method.} In the classical literature on Newton's method \cite{NoceWrig06}, local super-linear and quadratic rates of convergence for the pure Newton's method can be proved. In order to achieve global convergence, certain globalization strategies must be used such as line search or trust region approaches. Sometimes global convergence can be obtained while maintaining the local fast convergence \cite{Dennis}. There are other regularization strategies that add a positive definite matrix to the Hessian (like in the Levenberg--Marquardt method \cite{LEVENBERG,marquardt1963algorithm}) in order to stabilize Newton's method and to obtain global convergence thanks to having a positive definite metric. However, in general, this changes the Hessian and global convergence is obtained at the cost of loosing the local fast convergence rates. The pioneering work by Nesterov and Polyak \cite{nesterov2006cubic} has introduced cubic regularization as another globalization strategy that avoids line search and at the same time shows the standard fast local convergence rates of the pure Newton's method. In order to remedy the computational cost of cubic regularization, in \cite{mishchenko2023regularized,doikov2024gradient} a gradient regularization strategy was introduced that allows for global convergence with a super-linear rate of convergence for strongly convex problems. \emph{These works have fundamentally motivated the present paper and the development of our cubic and gradient regularized quasi-Newton methods that combines the best of both worlds: a global explicit super-linear rate of convergence while preserving a low computational cost per iteration by using only first-order information.}

\paragraph{Quasi-Newton methods.} Rodomanov and Nesterov \cite{rodomanov2021greedy} obtained the first local explicit super-linear convergence rate for greedy quasi-Newton methods and \cite{ye2023towards} obtained the first local explicit super-linear convergence rate for the SR1 quasi-Newton method, which both inspired our development in this paper. In subsequent works, \cite{rodomanov2022rates,jin2023non} proved the local non-asymptotic super-linear convergence rate for the classical Broyden family of quasi-Newton methods (Table~\ref{table:comparisonsmooth} and Table~\ref{tab:comparison}). 
However, all results above require the starting point to lie close to the optimal point to achieve the super-linear rate of convergence, unless additionally a line search strategy is used.  Therefore, it is worth looking for a global non-asymptotic super-linear convergence rate. The recent paper \cite{jiang2023online} showed a global non-asymptotic super-linear convergence rate by using an online learning strategy.
For BFGS with an exact line search strategy, \cite{jin2024non} obtain a global non-asymptotic convergence rate. The rate is linear at the beginning and super-linear when the iterates approach the local super-linear convergence region of the BFGS method. Later, \cite{rodomanov2024global,jin2024nonAmijo} achieve similar results with an inexact line search strategy, independently. 
\emph{In our paper, we combine the idea of regularization (cubic, gradient) and SR1 method and propose methods that are shown to converge globally with a global non-asymptotic super-linear rate without the help of globalization strategies (line search, trust region method).}

There are a few more works, appearing under similar names like `regularized Newton-type method' \cite{8804172,kanzow2022efficient,benson2018cubic,bianconcini2015use,lu2012trust,gould2012updating}. Since they do not consider non-asymptotic super-linear convergence rates and are using line search or trust region strategies, we do not describe them in more detail here.

\paragraph{Potential functions for convergence analysis.}
To obtain the non-asymptotic super-linear convergence, potential functions are widely used. We would like to mention several seminal papers on selecting potential functions (listed in Table~\ref{table:PotentialFunction}). \cite{broyden1973local} used the Frobenius-norm function for studying the Broyden family of quasi-Newton methods. Later,  an important contribution was made by Byrd, Liu, Nocedal and Yuan in \cite{byrd1987global}, who introduced the useful potential function $\sigma(A,G)=\trace(A^{-1}(G-A))-\log \det A^{-1}G$. \cite{rodomanov2022rates} used the first part of this function $\sigma(A,G)=\trace(A^{-1}(G-A))$ to show the convergence of the Broyden family of quasi-Newton methods on quadratic functions. A similar trace function $\trace(G-A)$ was first introduced in \cite{lin2021greedy}. In \cite{ye2023towards}, the authors also used the function $\trace(G-A)$ to study the quadratic case. However, we must emphasize that for their proofs, the existence of $A$ is crucial and necessary. That is the reason why we regard our potential function, despite its apparent similarity to the previous ones,, as a novel contribution. \emph{We propose the potential function $V(G)=\trace G$ which is much simpler than other potential functions while it also eases the proof of our non-asymptotic global super-linear convergence rate. Moreover, it provides a simple criterion for restarting the quasi-Newton method to guarantee the global convergence of our gradient regularized Algorithm~\ref{Alg:SR1_da_PQN2} (and its variant in Algorithm~\ref{Alg:SR1_da_PQN3}).}

\paragraph{Proximal quasi-Newton and Newton-type methods.}
One of the earliest work on proximal quasi-Newton methods is \cite{chen1999proximal}, where they adopted line search to guarantee global convergence and assume the Dennis--Mor\'e criterion to obtain asymptotic super-linear convergence.  \cite{lee2012proximal} also demonstrated an asymptotic super-linear local convergence rate of the proximal Newton and a proximal quasi-Newton method under the same condition. We believe that the Dennis--Mor\'e criterion is very restrictive and hard to check. Therefore, we rather use Lipschitz continuity of the Hessian instead. Later, \cite{scheinberg2016practical} showed that by a prox-parameter update mechanism, rather than a line search, they can derive a sub-linear global complexity. However, to the best of our knowledge, there is no result on non-asymptotic explicit super-linear convergence rates for a proximal quasi-Newton method. While the works discussed so far focus on the convergence rate, general convergence of so-called variable metric proximal gradient methods was also studied intensively. A crucial topic, which is mostly ignored, is the practical computation of the proximal mapping with respect to the metric that is induced by the quasi-Newton metric. Usually, simple proximal mappings in the standard metric become hard to solve when the metric is changed. There are a few works that are dedicated to this topic, proposing a proximal calculus tailored to exactly this situation. This was initiated in \cite{becker2012quasi} and consolidated in \cite{becker2019quasi}, where the interested reader finds an extensive literature review around this topic. Several follow up works have extended and improved these results \cite{scheinberg2016practical,karimi2017imro,becker2019quasi,wang2022inertial,kanzow2022efficient}. 
% \textcolor{red}{\st{\emph{We are aware of the practical importance of computing the proximal mapping when the metric changes and believe that these previous works can be extended to the current setting. Since, our contribution is already significant in the smooth setting, we leave this to future work.}}}

\section{Problem Setup and Main Results}
In this section, we present our notation, the considered problem setup with all standing assumptions, our proposed regularized proximal SR1-quasi-Newton algorithms and the main convergence results. We postpone the technical parts of the convergence analysis to Section~\ref{Section:conv}.

\subsection{Notation}
We first introduce our notations. We denote $\overline \R=\R\cup{\{+\infty\}}$. $\R^n$ is an Euclidean vector space of dimension $n$, which is equipped with the standard Euclidean norm $\vnorm{u}:=\sqrt{u^\top u}$, for any $u\in \R^n$.  Given a matrix $A\in\R^{n\times n}$, $\norm[]{A}$ denotes the matrix norm induced by the Euclidean vector norm. We write  $\trace A$ for the trace of the matrix $A\in\R^{n\times n}$. Moreover, $\opid$ denotes the identity matrix in $\R^n$. We adopt the standard definition of the Loewner partial order of symmetric positive semi-definite matrices. Let $A$ and $B$ be two symmetric matrices, we say $A\preceq B$ (or $A\prec B$) if and only if for any $u\in\R^n$, we have $u^\top(B-A)u\geq 0$ (or $u^\top(B-A)u> 0$, respectively).

\subsection{Problem set up}
In the whole paper, we consider the following optimization problem:
\begin{equation}\label{main:problem}
    \min_{x\in\R^n} F(x)\coloneqq g(x) +  f(x)\,,
\end{equation}
where we make the following assumptions:
\begin{assum}\label{ass:main-problem}
\begin{enumerate}
    \item $F= g+f$ is bounded from below;
    \item $\map{g}{\R^n}{\eR}$ is proper, lower semi-continuous (lsc), convex;
    \item $\map{f}{\R^n}{\R}$ is twice differentiable with $L$-Lipschitz gradient, $L_H$-Lipschitz Hessian and $f$ is $\mu$-strongly convex.
%    \item The Hessian of $f(x)$ is $L_H$-Lipschitz continuous. Namely,
% there exists \textcolor{black}{$L_H>0$} such that  for any $x,y\in \R^n$,  $\norm[]{\nabla^2f(x)-%\nabla^2f(y)}\preceq L_H\norm[]{x-y}$.
\end{enumerate}
\end{assum}

\subsection{Algorithms and main results}
Let us first recap the classical symmetric rank-1 quasi-Newton (SR1) update $G_+\coloneqq\text{SR1}(A,G,u)$. Given two $n\times n$ positive definite matrices $A$ and $G$ such that $0\prec A\preceq G$ and $u\in\R^n$, we define the SR1 update as follows:
\begin{equation}\label{eq:gen-SR1-update}
    \text{SR1}(A,G,u):=\begin{cases}
        G\,, &\ \text{if}\  (G-A)u=0\,;\\
        G-\frac{(G-A)uu^\top(G-A)}{u^\top(G-A)u}\,, & \ \text{otherwise}\,.
    \end{cases}
\end{equation}
In order to solve the problem in \eqref{main:problem}, we propose two regularized proximal SR1 quasi-Newton methods: Algorithms~\ref{Alg:SR1_da_PQN} and~\ref{Alg:SR1_da_PQN2}, where we present a variant of the latter as Algorithm~\ref{Alg:SR1_da_PQN3} in the appendix.\\

\textbf{Algorithm~\ref{Alg:SR1_da_PQN}} is inspired by the cubic regularized Newton method from \cite{nesterov2006cubic}, which is proved to converge globally. Step~\ref{alg:SR1_da_PQN-step1} of Algorithm~\ref{Alg:SR1_da_PQN} augments the classical quasi-Newton update step with a cubic regularization term that is scaled with the Lipschitz constant $L_H$ of the Hessian of $f$. Step~\ref{alg:SR1_da_PQN-step2} defines several variables, including the stationarity measure $F'(x_\kp)$ that is also used to terminate the algorithm in Step~\ref{alg:SR1_da_PQN-step6}. In Step~\ref{alg:SR1_da_PQN-step3}, we design a new correction strategy such that $\tilde G_k\succeq J_k$ (discussed in detail in Section~\ref{Section:conv}), where $J_k$ is defined in Step~\ref{alg:SR1_da_PQN-step4}. It is crucial to observe that this is only a theoretical quantity for the analysis of the algorithm. In practice, $J_k$ is only evaluated in the product with $u_k=x_\kp-x_k$ for which it simplifies to $J_ku_k=\nabla f(x_\kp)-\nabla f(x_k)$. Finally, Step~\ref{alg:SR1_da_PQN-step5} computes the SR1 update of the metric from \eqref{eq:gen-SR1-update}, which for $(\tilde G_k-J_k)u_k\neq 0$ (namely $F^\prime(x_\kp)\neq 0$), becomes $$ G_\kp = \text{SR1}(J_k,\tilde G_k, u_k) = G_k+\lambda_k\opid - \frac{\left((G_k+\lambda_k\opid)u_k-y_k\right)\left( (G_k+\lambda_k\opid)u_k-y_k\right)^\top}{u_k^\top( (G_k+\lambda_k\opid)u_k -y_k)}\,,$$
where $y_k:=\nabla f(x_\kp)-\nabla f(x_k)$.
\begin{algorithm}[H]
\caption{Cubic SR1 PQN}\label{Alg:SR1_da_PQN}
\begin{algorithmic}
\Require $x_0$, $G_0=L\opid$, $r_{-1}=0$.
\State \textbf{Update for $k = 0,\cdots,N$}:
\begin{enumerate}[1.]
\item\label{alg:SR1_da_PQN-step1} \textbf{Update} \begin{equation}\label{alg:update}
    x_\kp = \argmin_{x\in\R^n}\; \{ g(x)+ \scal{\nabla f(x_k)}{x-x_k}+\frac{1}{2}\norm[G_k+ L_Hr_\km\opid]{x-x_k}^2+{\frac{L_H}{3}\norm[]{x-x_k}^3}\}.
\end{equation}
    \item\label{alg:SR1_da_PQN-step2} \textbf{Set} $u_k=x_\kp-x_k$, $r_k=\norm[]{u_k}$, $\lambda_k= L_Hr_\km+L_Hr_k$. 
    \item\label{alg:SR1_da_PQN-step3} \textbf{Compute} $\tilde G_k=G_k+\lambda_k\opid$ and $$F^\prime(x_\kp)=\nabla f(x_\kp)-\nabla f(x_k)-\tilde G_{k}u_k\,.$$
    \item\label{alg:SR1_da_PQN-step4} \textbf{Denote but not use explicitly} $J_k\coloneqq\int_0^1\nabla^2f(x_k+tu_k)dt$.
    \item\label{alg:SR1_da_PQN-step5} \textbf{Compute} $G_\kp=\text{SR1}( J_k,\tilde G_k,u_k)$. 
    \item\label{alg:SR1_da_PQN-step6} \textbf{If $\norm{F^\prime(x_\kp)}=0$}:
    Terminate.
\end{enumerate}
\State \textbf{End}
\end{algorithmic}
\end{algorithm}
For the cubic regularized proximal quasi-Newton method in Algorithm~\ref{Alg:SR1_da_PQN}, we obtain the following global convergence with a non-asymptotic super-linear rate. 
\begin{theorem}[Main Theorem for Algorithm~\ref{Alg:SR1_da_PQN}]\label{thm:main-cubic}
For any initialization $x_0\in\R^n$ and any $N\in\N$, \textup{Cubic SR1 PQN}(Algorithm~\ref{Alg:SR1_da_PQN}) has a global convergence with the rate:
\begin{equation}\label{ineq1:super-linearrate}
    \norm[]{F^\prime (x_N)}\leq \left(\frac{C}{N^{1/2}}\right)^{N/2}\norm[]{F^\prime (x_0)} \,,
\end{equation}
where $C\coloneqq \left(2nL_HC_0 + nL\right)/\mu$ and $C_0\coloneqq \sqrt{\frac{ 2(F(x_0) -\inf F)}{\mu}}$.
\end{theorem}
\begin{proof}
    See Section~\ref{conv:alg1}.
\end{proof}
\begin{remark}\label{rem:cubic-main-rate}
    In fact, \eqref{ineq1:super-linearrate} indicates that the super-linear rate is only attained after a certain number of iterations $\bar N$ such that $\bar N\geq C^2$, where $C=O(\frac{n}{\mu^{3/2}})$.  $\bar N$ is comparable to $O(n^2)$ which is an important price to pay for globality. However, even for an initialization close to the optimal point, the super-linear rate of \cite{rodomanov2021greedy,ye2023towards} is attained after $\bar N=O(n)$ iterations. While the super-linear rate of \cite{jin2023non} is attained immediately, the radius of their local region of convergence is as small as $O(1/\sqrt{n})$ when $n$ is large. Besides, the local region criterion of \cite{jin2023non} requires the knowledge of the optimal point $x_*$, making this criterion useless in practice.
    
\end{remark}
\begin{remark}\label{rem:cubic-main-subgrad}
Following the optimality condition in \eqref{optimality condtion:cubic}, it turns out that $F^\prime (x)$ is a subgradient of the possibly non-smooth function $F$ at $x$. In turn, $\norm[]{F^\prime(x_N)}$ which goes to zero at a superlinear rate according to Theorem~\ref{thm:main-cubic}, quantifies how far $x_N$ is from optimality. 
\end{remark}

When $g\equiv 0$, Algorithm~\ref{Alg:SR1_da_PQN} is a cubic regularized SR1 quasi-Newton method for solving a smooth optimization problem\textcolor{black}{, and is also a novel contribution}. In this case, the global convergence rate in \eqref{ineq1:super-linearrate} can be stated explicitly in term of the gradient of the (smooth) objective
\[
    \norm[]{\nabla f (x_N)}\leq \left(\frac{C}{N^{1/2}}\right)^{N/2}\norm[]{\nabla f (x_0)} \,.
\]
Despite having the favorable rate of convergence, the computational complexity per iteration can be costly, since the cubic regularized subproblem that needs to be minimized in Step~\ref{alg:SR1_da_PQN-step3} usually does not have a closed form solution. \textcolor{black}{While the same is true also for cubic regularized Newton's method \cite{nesterov2006cubic}, the additional cost casts our algorithm inpractical for large-scale applications. In our implementation, we have adopted the strategy from \cite[Section 5]{nesterov2006cubic}, which proposes a reformulation into a convex one-dimensional subproblem and the usage of a bisectioning method. To remedy the high computational cost caused by the cubic regularization term}, inspired by \cite{mishchenko2023regularized,doikov2024gradient}, we provide two proximal quasi-Newton methods with gradient regularization (Algorithm~\ref{Alg:SR1_da_PQN2} and~\ref{Alg:SR1_da_PQN3}), which come with the same theoretical guarantees while maintaining a low computational cost per iteration.\\

\textbf{Algorithm~\ref{Alg:SR1_da_PQN2}} is derived from Algorithm~\ref{Alg:SR1_da_PQN}. Step~\ref{alg:SR1_da_PQN2-step1} of Algorithm~\ref{Alg:SR1_da_PQN2} equips the classical quasi-Newton update with a gradient regularization term instead of a cubic regularization. Step~\ref{alg:SR1_da_PQN2-step2} defines $J_k$ which has the same exclusively theoretical purpose as in Algorithm~\ref{Alg:SR1_da_PQN}. In Step~\ref{alg:SR1_da_PQN2-step3}, we define several variables, including $F^\prime(x_\kp)$ which is used to generate \textcolor{black}{the correction factor} $\lambda_\kp$ and terminate the algorithm in Step~\ref{alg:SR1_da_PQN2-step5}. Step~\ref{alg:SR1_da_PQN2-step4} is a new correction strategy. First, we compute $\hat G_\kp$ with a scaling parameter $1+\lambda_\kp$. Second, we set a restarting criterion based on the trace of $\hat G_\kp$. If the trace of $\hat G_\kp$ is bounded by $n\bar\kappa$ where $\bar\kappa\geq L$ is a hyper-parameter set at the beginning, we update the gradient regularized quasi-Newton metric $\tilde G_\kp$, otherwise, we restart the update of the quasi-Newton metric by setting $\tilde G_\kp=L\opid$.
\begin{algorithm}[H]
\caption{Grad SR1 PQN}\label{Alg:SR1_da_PQN2}
\begin{algorithmic}
\Require $x_0$, $G_0=L\opid$, $\lambda_0=0$, $\tilde G_0=G_0(1+\lambda_0)$, $\bar \kappa\geq L$\,. 
\State \textbf{Update for $k = 0,\cdots,N$}:
\begin{enumerate}[1.]
\item \label{alg:SR1_da_PQN2-step1}\textbf{Update} \begin{equation}\label{alg2:update}
    x_\kp =\argmin_{x\in\R^n} g(x) + \scal{\nabla f(x_k)}{x-x_k}+\frac{1}{2}\norm[\tilde G_k]{x-x_k}^2\,.
\end{equation}
    \item \label{alg:SR1_da_PQN2-step2}\textbf{Denote but not use explicitly} $J_k\coloneqq\int_0^1\nabla^2f(x_k+tu_k)dt$\,.
    \item \label{alg:SR1_da_PQN2-step3}\textbf{Compute} $u_k=x_\kp-x_k$, $r_k=\norm[]{u_k}$, and  
    \[
    \begin{split}
        F^\prime(x_\kp)=&\ \nabla f(x_\kp)-\nabla f(x_k)-\tilde G_{k}u_k\,,\\ 
        G_\kp =&\ \text{SR1}(J_k,\tilde G_k,u_k)\,,\\
        \lambda_\kp=&\ \frac{1}{\mu}\left(\sqrt{L_H\norm[]{F^\prime (x_\kp)}}+L_Hr_{k}\right)\,.
    \end{split}
    \]
    \item \label{alg:SR1_da_PQN2-step4}\textbf{Update $\tilde G_\kp$:} compute $\hat G_\kp= (1+\lambda_\kp)G_\kp$ (\textbf{Correction step}).
    \begin{enumerate}[]
        \item \textbf{If $\trace  \hat G_\kp\leq n\bar\kappa$:} we set $\tilde G_\kp= \hat G_\kp$\,.
        \item \textbf{Othewise:} we set $\tilde G_\kp= L\opid$ (\textbf{Restart step}).
    \end{enumerate}
    \item \label{alg:SR1_da_PQN2-step5} \textbf{If $\norm{F^\prime(x_\kp)}=0$}:
    Terminate.
\end{enumerate}
\State \textbf{End}
\end{algorithmic}
\end{algorithm}
In  Algorithm~\ref{Alg:SR1_da_PQN2}, $\lambda_k$ relies on the subgradient $F^\prime(x_k)$ of $F$ at $x_k$ which is different from what we did in Algorithm~\ref{Alg:SR1_da_PQN} and makes \eqref{alg2:update} easier to solve compared to the cubic regularization. The gradient regularization, despite being inspired by \cite{mishchenko2023regularized}, is tailored to the non-smooth composite problem. 
 The restarting strategy in Step~\ref{alg:SR1_da_PQN2-step4} from Algorithm~\ref{Alg:SR1_da_PQN2}, respectively, ensures that for any $k\geq 1$, the metric $\tilde G_k$ is bounded, namely, $\tilde G_k\preceq \trace\tilde G_k\opid\preceq n\bar\kappa \opid$ where $\bar\kappa \geq L$.

For  gradient regularized proximal quasi-Newton methods, we obtain global convergence with the following non-asymptotic super-linear rate.
\begin{theorem}[Main Theorem for Algorithm~\ref{Alg:SR1_da_PQN2}]\label{thm:main-grad}
For any initialization $x_0\in\R^n$ and any $N\in\N$, \textup{Grad SR1 PQN} (Algorithm~\ref{Alg:SR1_da_PQN2}) has a global convergence with the rate:
\begin{equation}\label{ineq2:super-linearrate}
    \norm[]{F^\prime (x_N)}\leq \left(\frac{C_{\mathrm{grad}}}{N^{1/4}}\right)^{N/2}\norm[]{F^\prime (x_0)} \,,
\end{equation}
where $C_{\mathrm{grad}}\coloneqq (n\bar\kappa\Theta + nL)/\mu$, $\Theta = \frac{1}{\mu}\left(L_HC_0 +\sqrt{L_H(L+n\bar\kappa)C_0}\right)$ and $C_0\coloneqq \sqrt{\frac{ 2(F(x_0) -\inf F)}{\mu}}$.
\end{theorem}
\begin{proof}
    See Section~\ref{conv:alg2}.
\end{proof}
Remark~\ref{rem:cubic-main-subgrad} also applies for Theorem~\ref{thm:main-grad}. 
\begin{remark}
    In fact, \eqref{ineq2:super-linearrate} indicates that the super-linear rate is only attained after $N\geq C_{\mathrm{grad}}^4$ number of iterations, where $C_{\mathrm{grad}}=O(\frac{n}{\mu^{5/2}})$. 
\end{remark}
The rate for gradient regularization is worse than for cubic regularization, since the super-linear rate is guaranteed to occur only after $N\geq C_{\mathrm{grad}}^4$ number of iterations. However, since the cubic regularization is no longer needed, the computational cost of the update step is significantly reduced. When $g\equiv0$, the convergence rate in \eqref{ineq2:super-linearrate} becomes 
\[
  \norm[]{\nabla f(x_N)}\leq \left(\frac{C_{\mathrm{grad}}}{N^{1/4}}\right)^{N/2}\norm[]{\nabla f(x_0)} \,,
\]
and, in this smooth case, the update step in Step~\ref{alg:SR1_da_PQN2-step1} is as follows:
 \begin{equation}
    x_\kp = x_k - \tilde G_k^{-1}\nabla f(x_k)\,,
\end{equation}
where the computational cost is, in fact, $O(n^2)$, since the use of $\tilde G_k=(1+\lambda_k)G_k$ in Step~\ref{alg:SR1_da_PQN2-step4}  enables the application of the Sherman--Morrison--Woodbury formula.  

When $g\not\equiv 0$, the update step has the same computational cost as that of a classical proximal quasi-Newton method and is much more tractable than that of Algorithm~\ref{Alg:SR1_da_PQN} which uses cubic regularization.\\

\textbf{Algorithm~\ref{Alg:SR1_da_PQN3}.} For the sake of completeness, in Section~\ref{sec:grad-reg-second-alg} in the appendix, we present a variant of Algorithm~\ref{Alg:SR1_da_PQN2}, which comes with the same convergence guarantees under a slight modification of the correction step, which is closer to the existing related works (e.g. \cite{kanzow2022efficient}) and has a smaller constant $C_{\mathrm{grad}}=O(\frac{n}{\mu^{3/2}})$ which shows a better dependence on the conditioning, however the computational cost of the update step in the smooth case is larger than that of Algorithm~\ref{Alg:SR1_da_PQN2}.

\section{Convergence Analysis}\label{Section:conv}
In this section, without loss of generality, we assume that $\norm[]{F^\prime(x_k)}>0$ for any $k\in\N$, since, otherwise, our algorithms terminate after a finite number of steps and our non-asymptotic rates would hold until then. 

Recall that throughout the whole convergence analysis, we assume that Assumption~\ref{ass:main-problem} holds and that all variables are defined either as in Algorithm~\ref{Alg:SR1_da_PQN} or~\ref{Alg:SR1_da_PQN2}. The convergence analysis of Algorithm~\ref{Alg:SR1_da_PQN3} can also be found in the appendix.

\subsection{Preliminaries}

First, we need several important properties of  $J_k\coloneqq\int_0^1\nabla^2f(x_k+tu_k)dt$, which is the same for both algorithms. 
\begin{lemma}
    For each $k\in\N$, we have $J_k\preceq L\opid$.
\end{lemma}
\begin{proof}
    Since $f$ has $L$-Lipschitz gradient, then, for any $x$, we have $\nabla^2 f(x)\preceq L\opid$. Therefore, due to the definition of $J_k$, we have $J_k\preceq L\opid$.  
\end{proof}
\begin{lemma}\label{lemma:JleqJ}
For each $k\in\N$, we have
    \begin{equation}
        \mu\opid\preceq J_k\preceq J_\km + L_H(r_k+r_\km)\opid\,.
    \end{equation}
    % and
    % \begin{equation}
    %     \frac{1}{(1+ \delta_k) }\tilde J_\km\preceq\tilde J_k\preceq (1+ \delta_k)\tilde J_\km\,,
    % \end{equation}
\end{lemma}
\begin{proof} 
Since $f(x)$ is $\mu$-strongly convex, the left first is trivial. Thus, we focus on the second inequality. The proof is similar with the one of \cite[Lemma 4.2]{rodomanov2021greedy} and \cite[Lemma 5]{ye2023towards}.  The assumption that the Hessian of $f$ is $L_H$-Lipschitz continuous means that there exists $L_H>0$ such that for any $x,y\in\R^n$, we have
\begin{equation}
    \norm[]{\nabla^2 f(x) - \nabla^2 f(y)}\leq L_H\norm[]{x-y}\,.
\end{equation}
For any $t\in[0,1]$ and $k\in\N$, letting $x=x_k+t u_k$ and $y=x_k-(1-t)u_\km$, we have that 
\begin{equation}\label{ineq:hessian}
    \norm[]{\nabla^2 f(x_k+tu_k) -\nabla^2 f(x_k-(1-t)u_k)}\leq L_H\norm[]{tu_k+(1-t)u_\km} \,.
    \end{equation}
Then, by the definition of $J_k$, we have
    \begin{equation}
        \begin{split}
            J_k-J_\km&=\int_0^1\nabla^2f(x_k+tu_k)dt-\int_0^1\nabla^2f(x_\km+tu_\km)dt\\
            &=\int_0^1\nabla^2f(x_k+tu_k)dt-\int_0^1\nabla^2f(x_k-(1-t)u_\km)dt\\
            &\preceq \int_0^1\norm[]{ \nabla^2f(x_k+tu_k)-\nabla^2f(x_k-(1-t)u_\km)}dt \cdot\opid\\
            &\preceq\int_0^1L_H\norm[]{tu_k+(1-t)u_\km}dt\cdot\opid\\
            &\preceq L_H(r_k+r_\km)\opid\,,
        \end{split}
    \end{equation}
    where the first inequality holds because of the definition of  the induced matrix norm, the second inequality holds due to \eqref{ineq:hessian} and the last one uses the triangle inequality.
    % Similarly, we have 
    % \begin{equation}
    %       \begin{split}
    %         J_\km-J_k&=-\int_0^1\nabla^2f(x_k+tu_k)dt+\int_0^1\nabla^2f(x_\km+tu_\km)dt\\
    %         &=\int_0^1-\nabla^2f(x_k+tu_k)dt+\int_0^1\nabla^2f(x_k-(1-t)u_\km)dt\\
    %         &\preceq \int_0^1(- \nabla^2f(x_k+tu_k)+\nabla^2f(x_k-(1-t)u_\km))dt \cdot\opid\\
    %         &\preceq\int_0^1L_H\norm[]{tu_k+(1-t)u_\km}dt\cdot\opid\\
    %         &\preceq L_H(r_k+r_\km)\opid.
    %     \end{split}
    % \end{equation}
\end{proof}
Then, we recall an important property of SR1 method, which shows that the update of SR1 method preserves the partial order of matrices.
\begin{lemma}\label{lemma:Anton}
    For any two symmetric positive definite matrices $A\in\R^{n\times n}$ and $G\in\R^{n\times n}$ with $0\preceq A\preceq G$ and any $u\in\R^n$, we observe the following for $G_+=\text{SR1}(A,G,u)$:
    \begin{equation}
        A\preceq G_+\preceq G\,.
    \end{equation}
\end{lemma}
\begin{proof}
    We recall the short proof from \cite[Part of Lemma 2.2]{rodomanov2021greedy} here for readers' convenience. If $(G-A)u=0$, the result is obvious. If $(G-A)u\neq 0$,  we conclude the statement as follows:
    \begin{equation}
        G_+-A = G-A - \frac{(G-A)uu^\top(G-A)}{u^\top (G-A)u} = (G-A)^{1/2}(1-\frac{\tilde u\tilde u^\top}{\tilde u^\top \tilde u})(G-A)^{1/2}\succeq 0\,,
    \end{equation}
    where $\tilde u = (G-A)^{1/2}u$.
\end{proof}
As motivated in the introduction, for any symmetric positive definite matrix $G$, we introduce a new potential function that is key to our convergence analysis:
\begin{equation}
    V(G)\coloneqq \trace G\,.
\end{equation}
Moreover, for any $G\succeq 0$ with $V(G)\geq 0$, given two matrices $A$ and $G$ satisfying $A\preceq G$, additionally, we introduce the following measurement function:
\begin{equation}
    \nu(A, G,u) = \begin{cases}
        0\,,&\ \text{if $(G-A)u=0$}\,;\\
        \frac{u^\top( G-A)(G-A)u}{u^\top (G-A) u}\,,&\ \text{otherwise}\,.
    \end{cases}
    \end{equation}
%$\nu(A,G,u)$ measures the distance between $A$ and $G$. 
Using the potential function $V(G)$ and $\nu(A,G,u)$, we deduce the following lemma showing that the SR1 update leads to a better approximation of $A$.
\begin{lemma}\label{lemma:potentialfunction}
Consider two symmetric positive definite matrices $A\in\R^{n\times n}$ and $ G\in\R^{n\times n}$ with $0\prec A\preceq  G$ and any $u\in\R^n$. Let $G_+=\text{SR1}(A,  G, u)$. Then, we have 
 \begin{equation}
    V(G) - V(G_+)=\nu(A, G,u)\,,
    \end{equation}
and 
\begin{equation}
    \nu(A, G,u) \geq  \frac{\norm[]{(G-A)u}^2}{u^\top G u} \,.\\
    \end{equation}
\end{lemma}
\begin{proof}
The result is trivial for the case $(G-A)u=0$, since this yields $G_+=G$. Therefore, let $(G-A)u\neq 0$. A simple calculation shows 
\[
        V(G)-V(G_+)=\trace(G-G_+)= \trace\left(\frac{ (G-A)uu^\top(G- A)^\top }{u^\top(G-A)u}\right)= \left( \frac{\norm[]{(G-A)u}^2}{u^\top(G-A)u} \right)\,,
\]
in which the last term can be bounded from below as claimed thanks to the ordering $G\succeq A\succ 0$.

% \color{gray}
% \begin{enumerate}
% \item If $( G-A)u=0$, it is obvious. If $( G-A)\neq 0$, we have
% \begin{equation}
%     \begin{split}
%         V(G)-V(G_+)&=\trace(G-G_+)= \trace\left(\frac{ (G-A)uu^\top(G- A)^\top }{u^\top(G-A)u}\right)= \left( \frac{\norm[]{(G-A)u}^2}{u^\top(G-A)u} \right)\\
%     \end{split}
% \end{equation}

%     \item If $( G-A)u=0$, it is obvious. We assume $( G-A)u\neq 0$. 
%     \begin{equation}
%     \begin{split}
%         V(G)-V(G_+)&=\trace(G-G_+)= \trace\left(\frac{ (G-A)uu^\top(G- A)^\top }{u^\top(G-A)u}\right)\geq \left( \frac{\norm[]{(G-A)u}^2}{u^\top G u} \right)\,/.\\
%     \end{split}
% \end{equation}
% where the last inequality holds since $G\succeq A\succ 0$.
%     \end{enumerate}
\end{proof}

\subsection{Convergence analysis of Algorithm~\ref{Alg:SR1_da_PQN}}\label{conv:alg1}

Let us focus on Algorithm~\ref{Alg:SR1_da_PQN}. 
% In Algorithm~\ref{Alg:SR1_da_PQN}, the main update step is
% \begin{equation}
%     x_\kp = \argmin \left\{g(x) + \scal{\nabla f(x_k)}{x-x_k}+\frac{1}{2}\norm[G_k+L_Hr_{k-1}\opid]{x-x_k}^2+\frac{L_H}{3}\norm[]{x-x_k}^3\,\right\}\,.
% \end{equation}
The optimality condition of the update step in \eqref{alg:update} is
\begin{equation}\label{optimalitycond}
    \partial g(x_\kp)+\nabla f(x_k) + (G_k+ L_Hr_\km)(x_\kp-x_k) + L_H\norm[]{x_\kp-x_k}(x_\kp-x_k)\ni0\,,
\end{equation}
which, using our notations, simplifies to 
\begin{equation}\label{optimality condtion:cubic}
    \partial g(x_\kp)+ \nabla f(x_k) + ( G_k+\lambda_k)(x_\kp-x_k)\ni0\,.
\end{equation}
We denote $v_\kp \coloneqq -  \nabla f(x_k) - (G_k+\lambda_k)(x_\kp-x_k)$ and  $F^\prime(x_k)\coloneqq v_k+\nabla f(x_k)$ and obtain $F^\prime(x_k)\in \partial F(x_k)$.\\

The following lemma is important for the correction step from $G_k$ to $\tilde G_k$, showing that for each $k$, we have $\tilde G_k\succeq J_k$.
\begin{lemma}\label{lemma:GgeqJ}
    For any $k\in\N$, we have that $\tilde G_k=G_k+L_H(r_k+r_\km)\opid\succeq G_\kp\succeq J_k$.
\end{lemma}
\begin{proof}
We prove the result by induction. For $k=0$, by definition, we have $ J_0 \preceq  L\opid\preceq \tilde G_0$, and using Lemma~\ref{lemma:Anton}, we have $J_0\preceq G_1=\text{SR1}(J_0,\tilde G_0,u_0)\preceq \tilde G_0$. Now, we assume that $J_\km\preceq G_k\preceq \tilde G_\km$ holds for $k-1$. Then, for $k$, we have
    % \begin{equation}
    %     \tilde G_k=G_k+ L_H(r_k+r_\km)\opid\succeq J_\km+L_H(r_\km+r_k)\opid\succeq J_k\,,
    % \end{equation}
    \begin{equation}
        J_k\preceq J_\km+L_H(r_\km+r_k)\opid\preceq G_k+ L_H(r_k+r_\km)\opid=\tilde G_k\,,
    \end{equation}
    where the first inequality holds due to Lemma~\ref{lemma:JleqJ}, which shows that $0\preceq J_k \preceq \tilde G_k$ and we conclude the induction by Lemma~\ref{lemma:Anton}:
    \[
      J_k\preceq G_\kp=\text{SR1}(J_k,\tilde G_k, u_k)\preceq \tilde G_k\,. \qedhere
    \]
\end{proof}
Now, we study the descent lemma of function values.
\begin{lemma}\label{lemma:fdescent}
For any $k\in\N$, we have
\begin{equation}
    F(x_\kp)-F(x_k)\leq -\frac{\mu}{2}r_k^2\,.
\end{equation}
\end{lemma}
\begin{proof} We start with the strong convexity of the smooth function $f$: for any $x$ and $y$, we have 
\begin{equation}
    f(x)-f(y)\leq \scal{\nabla f(x)}{x-y} - \frac{\mu}{2}\norm[]{x-y}^2\,.
\end{equation}
Letting $x=x_\kp$ and $y=x_k$, we have
    \begin{equation}\label{cubic decentlemma ineq:1}
        \begin{split}
           f(x_\kp)-f(x_k)&\leq \scal{\nabla f(x_\kp)}{x_\kp-x_k} - \frac{\mu}{2}\norm[]{x_\kp-x_k}^2\\
            &= \scal{\nabla f(x_k)}{x_\kp-x_k}+\scal{\nabla f(x_\kp)-\nabla f(x_k)}{x_\kp-x_k} - \frac{\mu}{2}\norm[]{x_\kp-x_k}^2\\
            &=\scal{\nabla f(x_k)}{x_\kp-x_k}+\scal{J_k(x_\kp-x_k)}{x_\kp-x_k} - \frac{\mu}{2}\norm[]{x_\kp-x_k}^2\,,
            % \text{Lemma~\ref{lemma:JleqJ}}&\leq \scal{\nabla f(x_k)}{x_\kp-x_k}+\scal{J_\km(x_\kp-x_k)}{x_\kp-x_k}
            % +L_H(r_k+r_\km)\norm[]{x_\kp-x_k}^2\\
            % & - \frac{\mu}{2}\norm[]{x_\kp-x_k}^2\\
            % \text{Lemma~\ref{lemma:GgeqJ}}&\leq \scal{\nabla f(x_k)}{x_\kp-x_k}+\scal{G_k(x_\kp-x_k)}{x_\kp-x_k}
            % +L_H(r_k+r_\km)\norm[]{x_\kp-x_k}^2\\
            % & - \frac{\mu}{2}\norm[]{x_\kp-x_k}^2\,.\\
        \end{split}
    \end{equation}
    where the last equality holds since $J_ku_k=\int_0^1 \nabla^2 f(x_k+tu_k)u_kdt = \nabla f(x_\kp)-\nabla f(x_k)$.
    According to Lemma~\ref{lemma:JleqJ}, we have $\scal{J_ku_k}{u_k}\leq \scal{J_\km u_k+L_H(r_k+r_\km)u_k}{u_k} $. Thus, 
     \eqref{cubic decentlemma ineq:1} implies that
     \begin{equation}\label{cubic decentlemma ineq:2}
      \begin{split}
          f(x_\kp)-f(x_k)&\leq \scal{\nabla f(x_k)}{x_\kp-x_k}+\scal{J_\km(x_\kp-x_k)}{x_\kp-x_k}\\&+L_H(r_k+r_\km)\norm[]{x_\kp-x_k}^2- \frac{\mu}{2}\norm[]{x_\kp-x_k}^2\,.
      \end{split}
     \end{equation}
    By Lemma~\ref{lemma:GgeqJ}, we have $G_k\succeq J_\km$ and thus $\scal{G_ku_k}{u_k}\geq \scal{J_\km u_k}{u_k}$ for any $k\geq 1$. Therefore, \eqref{cubic decentlemma ineq:2} implies that 
    \begin{multline}\label{cubic decentlemma ineq:3}
          f(x_\kp)-f(x_k)\leq \scal{\nabla f(x_k)}{x_\kp-x_k}+\scal{G_k(x_\kp-x_k)}{x_\kp-x_k} \\
          +L_H(r_k+r_\km)\norm[]{x_\kp-x_k}^2- \frac{\mu}{2}\norm[]{x_\kp-x_k}^2\,.
    \end{multline}
    Since $x_\kp$ is the solution of the inclusion \eqref{optimality condtion:cubic} and $g$ is convex, the subgradient inequality holds:
    \begin{equation}\label{cubic decentlemma ineq:4}
        g(x_\kp) -g(x_k)\leq \scal{v_\kp}{x_\kp-x_k}\,.
    \end{equation}
    Combining \eqref{cubic decentlemma ineq:3} and \eqref{cubic decentlemma ineq:4}, we deduce that 
    \begin{equation}
    \begin{split}
        F(x_\kp) -F(x_k)&\leq \scal{v_\kp+ \nabla f(x_k)}{x_\kp-x_k}+\scal{G_k(x_\kp-x_k)}{x_\kp-x_k}
            \\&\quad+L_H(r_k+r_\km)\norm[]{x_\kp-x_k}^2 - \frac{\mu}{2}\norm[]{x_\kp-x_k}^2\\
           &=-\frac{\mu}{2}r_k^2\,,  \\
    \end{split}
    \end{equation}
    where the last equality holds due to \eqref{optimality condtion:cubic}.
\end{proof}
Due to the above descent lemma of function values, we can derive bounds for several sequences.
\begin{lemma}\label{lemma:sumincrease}
Given a number of iterations $N\in\N$, we have: 
%%\begin{enumerate}[1.]
%%    \item \begin{equation}
%%    \sum_{k=0}^\infty r_k^2\leq \frac{2(F(x_0) -\inf F)}{\mu}<+\infty\,,
%%\end{equation}
%%    \item \begin{equation}
%%    \sum_{k=0}^{N-1} r_k\leq C_0 N^\frac{1}{2}\,,
%%\end{equation}
%%    \item \begin{equation}
%%    \sum_{k=1}^{N} r_k\leq C_0 N^\frac{1}{2}\,,
%%\end{equation}
%%\end{enumerate}
    \begin{gather}
    \sum_{k=0}^{N} r_k^2\leq \frac{2(F(x_0) -\inf F)}{\mu}<+\infty\,, \\
    \sum_{k=0}^{N-1} r_k\leq C_0 N^\frac{1}{2}\,, \\
    \sum_{k=1}^{N} r_k\leq C_0 N^\frac{1}{2}\,,
\end{gather}
where $C_0\coloneqq \sqrt{\frac{ 2(F(x_0) -\inf F)}{\mu}}$.
\end{lemma}
\begin{proof}
    The first inequality is derived from Lemma~\ref{lemma:fdescent}. Then, we apply the Cauchy--Schwarz inequality and obtain
    \begin{equation}
        \sum_{i=0}^{N-1} r_k \leq \left(\sum_{k=0}^{N-1} r_k^2\right)^{1/2}\left(\sum_{i=0}^{N-1} 1^{\frac{2}{1}}\right)^{1/2}\leq\left(\sum_{k=0}^{\infty} r_k^2\right)^{1/2}N^{1/2}\leq C_0N^{1/2} \,.
    \end{equation}
    Similarly, we have 
    \begin{equation}
        \sum_{i=1}^{N} r_k \leq \left(\sum_{k=1}^{N} r_k^2\right)^{1/2}\left(\sum_{i=1}^{N} 1^{\frac{2}{1}}\right)^{1/2}\leq \left(\sum_{k=0}^{\infty} r_k^2\right)^{1/2}N^{1/2}\leq C_0N^{1/2} \,. \qedhere
    \end{equation}
    % \begin{equation}
    %     \sum_{i=0}^{N-1} r^2_k \leq (\sum_{k=0}^{N-1} r_k^3)^{2/3}(\sum_{i=0}^{N-1} 1^{3})^{1/3}\leq(\sum_{k=0}^{\infty} r_k^3)^{1/3}N^{1/3}\leq C_0 N^{1/3} \,.
    % \end{equation}
\end{proof}

Now, we provide several simple but useful inequalities showing the relations among $G_k$, $\tilde G_k$, $J_k$.
% \begin{lemma}\label{lemma:JJGG}
% For any $k\in\N$, we have
%     \begin{enumerate}
%         \item $G_\kp\leq \tilde G_\kp$;
%        % \item $\tilde G_\kp^{-1}\leq G_\kp^{-1}$;
%        % \item $\tilde G_k\leq (1+\frac{\lambda_k}{\mu})G_k$;
%        % \item $\frac{J_\km}{1+\delta_k}\leq J_k\leq (1+\delta_k)J_\km$ where $\delta_k=(1+\frac{Mr_k}{2})(1+\frac{Mr_\km}{2})-1$ and $M=\frac{L_H}{\mu^{3/2}}$.
%       
%     \end{enumerate}
% \end{lemma}
% \begin{proof}
%     \begin{enumerate}
%         \item $G_\kp\preceq G_\kp + \lambda_k\opid$;
%        % \item  Due to $G_\kp\leq \tilde G_\kp$;
%        % \item $G_\kp\leq \tilde G_\kp= G_\kp + \lambda_k\opid\preceq G_\kp + \frac{\lambda_k}{\mu} G_\kp$;
%         %\item See Anton's paper \cite[Lemma 4.2]{rodomanov2021greedy}.
%     \end{enumerate}
% \end{proof}
Based on the growth rate of the summation over sequence $(r_k)_{k\in\N}$, we can derive an upper bound growth rate of the sequence $(\lambda_k)_{k\in\N}$.
\begin{lemma}\label{lemma:growth}
Given the number of iterations $N\in\N$, we have \[ \sum_{k=1}^{N}\lambda_k\leq 2L_HC_0N^{1/2} \] with $C_0$ as in Lemma~\ref{lemma:sumincrease}.
    % \begin{enumerate}
    %     \item $\sum_{k=0}^{N-1}\lambda_k=2L_H\sum_{k=0}^{N-1}(r_k+r_\km)\leq 4L_HC_0N^{1/2}$;
    %     %\item $\sum_{k=0}^{N-1}\delta_k=(M+\frac{M^2}{4})C_0N^{1/2}$
    % \end{enumerate}
\end{lemma}
\begin{proof}
The result follows by combining the definition of $\lambda_k$ with Lemma~\ref{lemma:sumincrease}.
    % Combining the definition of $\lambda_k$ with the result from Lemma~\ref{lemma:sumincrease} directly yields \PO{Where does the factor 2 come from in the definition of $\lambda_k$? Looks like the inequality should be better by this factor, so $2$ instead of $4$ in the end. If I am right, I think the proof can be shortened to ``The result follows by combining the definition of $\lambda_k$ with Lemma~\ref{lemma:sumincrease}.'' (without showing the calculation.)} 
    % \[
    % \sum_{k=1}^{N}\lambda_k=L_H\sum_{k=1}^{N}(r_k+r_\km)\leq 2L_HC_0N^{1/2} \,.\qedhere
    % \]
        % \item \begin{equation}
        % \begin{split}
        %     \sum_{k=0}^{N-1}\delta_k& = \frac{M}{2}\sum_{k=0}^{N-1}(r_k+r_\km)+\frac{M^2}{4}\sum_{k=0}^{N-1}r_kr_\km\\
        %     &\leq \frac{M}{2}\sum_{k=0}^{N-1}(r_k+r_\km)+\frac{M^2}{8}\sum_{k=0}^{N-1}(r_k^2+r_\km^2)\\
        %     & \leq MC_0N^{1/2}+\frac{M^2}{4}C_0\\
        %     & \leq(M+\frac{M^2}{4})C_0N^{1/2}\,.\\
        % \end{split}
        % \end{equation}
    
\end{proof}
% \begin{lemma}\label{lemma:potentialfunction}
% \begin{enumerate}
%     \item 
%     \begin{equation}
%     V(J_k,\tilde G_k) - V(J_k,G_\kp)=\log(1+\nu(J_k,\tilde G_k,u_k)^2)\,.
% \end{equation}
% \item 
% \begin{equation}
%     \nu(J_k,\tilde G_k,u_k)^2 \geq \frac{U_k^\top(\tilde G_k-J_k)\tilde G_\kp^{-1}(\tilde G_k-J_k)u_k}{u_k^\top \tilde G_k u_k}>0\,.
% \end{equation}
% \end{enumerate}
% \end{lemma}
% \begin{proof}
% \begin{enumerate}
% \item \cite[Lemma~2]{ye2023towards};
%     \item Inspired by \cite[Lemma~3]{ye2023towards},
%     \begin{equation}
%     \begin{split}
%         \nu(J_k,\tilde G_k,u_k)^2& \geq \frac{U_k^\top(\tilde G_k-J_k) G_\kp^{-1}(\tilde G_k-J_k)u_k}{u_k^\top \tilde G_k u_k} \\
%         (\text{Lemma~\ref{lemma:JJGG}})& \geq \frac{U_k^\top(\tilde G_k-J_k) \tilde G_\kp^{-1}(\tilde G_k-J_k)u_k}{u_k^\top \tilde G_k u_k} \\
%     \end{split}
% \end{equation}
% \end{enumerate}
% \end{proof}
\begin{lemma}\label{ccubic:ukGu}
For any $k\in \N$, we have 
\begin{equation}
    u_k^\top \tilde G_k u_k\leq \frac{1}{\mu}\norm[]{F^\prime (x_k)}^2\,.
\end{equation}
\end{lemma}

\begin{proof}
    Since $g$ is convex, hence, its subdifferential monotone, we have $u_k^\top(v_\kp-v_k)\geq 0$. 
% Then, there exists a symmetric positive semidefinite matrix $B_k\coloneqq \frac{(v_\kp-v_k)(v_\kp-v_k)^\top}{ u_k^\top(v_\kp-v_k)} $ such that $B_ku_k =(v_\kp-v_k)$.
According to \eqref{optimality condtion:cubic}, we obtain that $\tilde G_ku_k=-\nabla f(x_k)-v_\kp$. Then we can deduce that:
    \begin{equation}
    \begin{split}
         u_k^\top \tilde G_k u_k
         &\leq u_k^\top \tilde G_k u_k+ u_k^\top (v_\kp-v_k)\\
         &=u_k^\top\left( -\nabla f(x_k)-v_\kp+v_\kp-v_k   \right)\\
         &= -u_k^\top F^\prime(x_k)\\
         &=-u_k^\top \tilde G_k^{1/2}\tilde G_k^{-1/2}F^\prime(x_k)\\
         &\leq \norm[\tilde G_k]{u_k}\norm[\tilde G_k^{-1}]{F^\prime(x_k)}  \\
         &\leq  \frac{1}{2}u_k^\top \tilde G_k u_k +\frac{1}{2}\norm[\tilde G_k^{-1}]{F^\prime(x_k)}^2\\
         &\leq  \frac{1}{2}u_k^\top \tilde G_k u_k +\frac{1}{2\mu}\norm[]{F^\prime(x_k)}^2\,,
    \end{split}
    \end{equation}
    where the second inequality uses Cauchy-Schwarz inequality, the third inequality uses Young's inequality and the last inequality holds since  $\tilde G_k^{-1}\preceq \frac{1}{\mu}\opid$ (see Lemma~\ref{lemma:GgeqJ} and Lemma~\ref{lemma:JleqJ}). By rearranging the above inequality, we obtain the desired result.
\end{proof}
% \begin{proof}
% Since $g$ is convex, we have $u_k^\top(v_\kp-v_k)\geq 0$. Then, there exists a symmetric positive semidefinite matrix $B_k\coloneqq \frac{(v_\kp-v_k)(v_\kp-v_k)^\top}{ u_k^\top(v_\kp-v_k)} $ such that $B_ku_k =(v_\kp-v_k)$. According to \eqref{optimality condtion:cubic}, we obtain that $(\tilde G_k+B_k)u_k=-\nabla f(x_k)-v_\kp+v_\kp - v_k=-F^\prime(x_k)$. since $\tilde G_k\succeq J_k\succeq\mu\opid$ and $B_k\succeq 0$, we have $(\tilde G_k+B_k)\succeq\mu\opid$ and $(\tilde G_k+B_k)^{-1}\preceq\frac{1}{\mu}\opid$.
% Therefore, we can deduce that:
%     \begin{equation}
%     \begin{split}
%         u_k^\top \tilde G_k u_k&\leq u_k^\top \tilde G_k u_k+ u_k^\top (v_\kp-v_k)\\&=u_k^\top(\tilde G_k+B_k)u_k\\&= u_k^\top(\tilde G_k+B_k)(\tilde G_k+B_k)^{-1}(\tilde G_k+B_k)u_k\\&\leq \frac{1}{\mu}\norm[]{F^\prime(x_k)}^2\,.
%     \end{split}
%     \end{equation}
%     % where the last inequality holds since $\tilde G_k\succeq J_k\succeq\mu\opid$ and $B_k\succeq 0$.
% \end{proof}

The following lemma is a crucial estimate of the descent of the corrected SR1 metric $\tilde G_k$ with respect to our potential function $V(\tilde G_k)=\trace \tilde G_k$.
%We recall the notation $V(G)\coloneqq \trace G$. Then for function $V(\tilde G_k)$ we can derive the following descent lemma.
\begin{lemma}\label{lemma:maindescent}
For any $k\in\N$, we have the following descent inequality:
\begin{equation}\label{lemma1:VminusV}
    V(\tilde G_k)-V(\tilde G_\kp)\geq \frac{\mu g^2_\kp}{g_k^2}-n\lambda_\kp\,,
\end{equation}
where $g_k\coloneqq \norm[]{F^\prime(x_k)}$.
\end{lemma}
\begin{proof}
By the optimality condition \eqref{optimality condtion:cubic}, we have $\tilde G_k u_k+v_\kp= -\nabla f(x_k)$ and by definition of $J_k$, we have $J_ku_k=\nabla f(x_\kp)-\nabla f(x_k)$. Using Lemma~\ref{lemma:potentialfunction}, we make the following estimation:
    \begin{equation}
    \begin{split}
         V(\tilde G_k) - V(G_\kp) 
         = \nu(J_k,\tilde G_k,u_k) 
         \geq&\ \frac{\vnorm{(\tilde G_k-J_k)u_k}^2}{u_k^\top \tilde G_k u_k}
         =\ \frac{\vnorm{\nabla f(x_\kp)+v_\kp}^2}{u_k^\top \tilde G_k u_k} \\
         =&\ \frac{\vnorm[]{F^\prime (x_\kp)}^2}{u_k^\top\tilde G_k u_k}
          \geq \frac{\mu\norm[]{F^\prime (x_\kp)}^2}{\norm[]{F^\prime(x_k)}^2}
          =\frac{\mu g_\kp^2}{g_k^2}\,,
    \end{split}
    \end{equation}
    where the second inequality holds due to Lemma~\ref{ccubic:ukGu}. Summing this with 
    \begin{equation}\label{cubic ineq:V1}
        V(G_\kp)-V(\tilde G_\kp) = \trace(G_\kp-G_\kp-\lambda_\kp\opid) \geq - n\lambda_\kp\,,
    \end{equation}
    we obtain the desired inequality. 

% \PO[inline]{The following in gray is your proof.}
% \color{gray}
% Together with second result of Lemma~\ref{lemma:potentialfunction}, we have
%     \begin{equation}
%     \begin{split}
%          \textcolor{blue}{\nu(J_k,\tilde G_k,u_k)} &\geq \frac{u_k^\top(\tilde G_k-J_k)(\tilde G_k-J_k)u_k}{u_k^\top \tilde G_k u_k}\\& = \frac{(-\nabla f(x_\kp)-v_\kp)(-\nabla f(x_\kp)-v_\kp)}{u_k^\top \tilde G_k u_k}\\
%          &= \frac{\norm[]{F^\prime (x_\kp)}^2}{u_k^\top\tilde G_k u_k}\\
%           &\geq \frac{\mu\norm[]{F^\prime (x_\kp)}^2}{\norm[]{F^\prime(x_k)}^2}\\&=\frac{\mu g_\kp^2}{g_k^2}\,,
%     \end{split}
%     \end{equation}
%     where the second inequality holds due to \text{Lemma~\ref{ccubic:ukGu}}.
%     Together with the first result of Lemma~\ref{lemma:potentialfunction}, we obtain
%     \begin{equation}\label{cubic ineq:V3}
%         V(\tilde G_k) - V(G_\kp)=\textcolor{blue}{\nu(J_k,\tilde G_k,u_k)}\geq\frac{\mu g_\kp^2}{g_k^2} \,.
%     \end{equation}
    
%     We also need the following inequality.
%     \begin{equation}\label{cubic ineq:V1}
%         V(G_\kp)-V(\tilde G_\kp) = \trace(G_\kp-G_\kp-\lambda_\kp\opid) \geq - n\lambda_\kp\,.
%     \end{equation}
%     Summing \eqref{cubic ineq:V1} and \eqref{cubic ineq:V3},  we obtain the desired inequality. 
% \color{black}    
\end{proof}

Now, we are ready to prove our main theorem for Algorithm~\ref{Alg:SR1_da_PQN}.
%%\begin{theorem}[Main Theorem]
%%For any initialization $x_0\in\R^n$ and any $N\in\N$, cubic SR1 PQN has a global convergence with the rate:
%%\begin{equation}\label{ineq:super-linearrate}
%%    \norm[]{F^\prime (x_N)}\leq \left(\frac{C}{N^{1/2}}\right)^{N/2}\norm[]{F^\prime (x_0)} \,,
%%\end{equation}
%%where $C\coloneqq \left(4nL_HC_0 + V(G_0)\right)/\mu$ and $C_0\coloneqq \sqrt{\frac{ 2(F(x_0) -\inf F)}{\mu}}$.
%%\end{theorem}
\begin{proof}[\textbf{Proof of Theorem~\ref{thm:main-cubic}.}]
For simplicity, we denote $\gamma_k^2 \coloneqq\frac{g^2_\kp}{g^2_k}$, where $g_k\coloneqq \norm[]{F^\prime(x_k)}$.
    By Lemma~\ref{lemma:maindescent}, summing \eqref{lemma1:VminusV} from $k=0$ to $N$, we obtain
    \begin{equation}
        V(\tilde G_0)-V(\tilde G_N)\geq \mu\sum_{k=0}^{N-1} \gamma_k^2 - n\sum_{k=1}^{
        N} \lambda_k\,.
    \end{equation}
    Since, for every $k$, our method preserves the relation $\tilde G_k\succeq G_\kp\succeq J_k$ (see Lemma~\ref{lemma:GgeqJ}),  we have $V(\tilde G_k)>0$ for all $k\in\N$. Therefore, we deduce that
    \begin{equation}\label{ineq:sumineq}
    \begin{split}
    V(\tilde G_0)&\geq \mu \sum_{k=0}^{N-1} \gamma_k^2 - n\sum_{k=1}^{
        N} \lambda_k\,.\\
    \end{split}
    \end{equation}
    By Lemma~\ref{lemma:growth}, we obtain 
    \begin{equation}\label{mainstep}
    CN^{1/2}\geq \frac{1}{\mu}\left(V(G_0)+ 2nL_HC_0N^{1/2}\right)\geq \sum_{k=0}^{N-1}\gamma_k^2\,.
    \end{equation}
    where $C\coloneqq (2nL_HC_0 + V(G_0))/\mu$. Dividing \eqref{mainstep} by $N$ on both sides, 
    we deduce with the concavity and monotonicity of $\log x$ that 
    \begin{equation}
    \begin{split}
          \log\frac{C}{N^{1/2}}&\geq \log\left(\frac{1}{N}\sum_{k=0}^{N-1}\gamma_k^2\right)\geq\frac{1}{N}\sum_{k=0}^{N-1}\log(\gamma_k^2)\geq \log \left(\left(\prod_{k=0}^{N-1}\frac{g_\kp^2}{g_k^2}\right)^{1/N}\right)=\log\left(\left(\frac{g_N}{g_0}\right)^{2/N}\right)\,.\\
    \end{split}
    \end{equation}
    Thus, we obtain the result
    \begin{equation}
        g_N\leq \left(\frac{C}{N^{1/2}}\right)^{N/2}g_0\,.    \qedhere
    \end{equation}
\end{proof}

\subsection{Convergence analysis of Algorithm~\ref{Alg:SR1_da_PQN2}}\label{conv:alg2}

Now, let us analyze Algorithm~\ref{Alg:SR1_da_PQN2}.
%  In Algorithm~\ref{Alg:SR1_da_PQN2}, the main step is
% \begin{equation}%\label{optimality condition:grad}
%     x_\kp =\argmin_{x\in\R^n} \left\{g(x) + \scal{\nabla f(x_k)}{x-x_k}+\frac{1}{2}\norm[\tilde G_k]{x-x_k}^2\right\}\,,
% \end{equation}
% where $\lambda_k = \frac{1}{\mu}\left(L_Hr_\km+\sqrt{L_H\norm[]{F^\prime(x_k)}}\right)$. 
The optimality condition of the update step in \eqref{alg2:update} is 
\begin{equation} \label{optimality condition:grad}   
\nabla f(x_k) + v_\kp+ \tilde G_k(x_\kp-x_k)=0\,,
\end{equation}
where $v_\kp\coloneqq -\nabla f(x_k) -\tilde G_k(x_\kp-x_k) \in \partial g(x_\kp)$.

As in the previous section, we need several bounds for sequences $(\tilde G_k)_{k\in\N}$, $( G_k)_{k\in\N}$ and $(J_k)_{k\in\N}$.
\begin{lemma}\label{lemma2:bound}
For each $k\in\N$, we have 
% \begin{equation}
%     n\bar\kappa\opid\succeq\tilde G_k\succeq G_{k+1}\succeq J_k\succeq\mu\opid\,.
% \end{equation}
\begin{gather}\label{eq:lemma2:bound}
    \mu\opid\preceq J_k\preceq G_{k+1}\preceq \tilde G_k\,,\\
    \trace G_{k+1}\leq \trace \tilde G_k\leq n\bar\kappa\,.
\end{gather}
% \begin{equation}\label{eq:lemma2:bound}
%     \mu\opid\preceq J_k\preceq G_{k+1}\preceq \tilde G_k\,,
% \end{equation}
% \textcolor{blue}{and 
% \begin{equation}\label{eq:lemma2:bound}
%     \trace G_{k+1}\leq \trace \tilde G_k\leq n\bar\kappa\,.
% \end{equation}
% }
% \begin{enumerate}
%     \item $\tilde G_k\succeq J_k\succeq\mu\opid$,
%     \item $n\bar\kappa\opid\succeq G_{k+1}\succeq J_k$.
%     %\item  $2L_Hr_\km+2L_Hr_k\leq \lambda_k$.
% \end{enumerate}
\end{lemma}
\begin{proof}
We first notice that since $g$ is convex, the monotonicity of its subdifferential yields $$u_k^\top({v_\kp-v_k})\geq 0\,,$$ for any $k\in\N$. Then, for any $k\in\N$, we have
\begin{equation}\label{grad ineq:ukF}
    u_k^\top \tilde G_k u_k\leq u_k^\top \tilde G_k u_k + u_k^\top (v_\kp-v_k)=u_k^\top(-\nabla f(x_k) -v_\kp+v_\kp - v_k) =- u_k^\top F^\prime (x_k) \,.
\end{equation}
The first inequality in \eqref{eq:lemma2:bound} is a direct consequence of the strong convexity of $f$. 

Now, we prove the remaining inequalities in \eqref{eq:lemma2:bound} by induction. In order to validate the base case $k=0$, we first observe that $\tilde G_0=G_0(1+\lambda_0)\succeq J_0$ holds by Lipschitz continuity of $\nabla f$ and $G_0=LI$ and $\lambda_0=0$. Then, Lemma~\ref{lemma:Anton} shows the desired property:
\[
J_0\preceq G_1=\text{SR1}(J_0,\tilde G_0,u_0)\preceq\tilde G_0\preceq nL\opid\preceq n\bar\kappa\opid\,.
\]
For showing the induction, we suppose now that \eqref{eq:lemma2:bound} holds for $k=N-1$. %Then, $G_N\succeq \mu\opid$ since $J_{N-1}\succeq\mu\opid$. Here $\hat G_{N} = G_{N}(1+\lambda_{N})$.
%   
%Now we are going to  show that for $k=N$, these results hold.
%
We discuss by cases:
\begin{enumerate}[1.]
        \item If $\trace \hat G_{N}\leq n\bar\kappa$, then, $\tilde G_N=\hat G_N$.  Since $\tilde G_N=G_N(1+\lambda_N)\geq \mu\lambda_N $ and Cauchy--Schwarz inequality, the above equality \eqref{grad ineq:ukF} implies
\begin{equation}
    \mu\lambda_N \norm[]{u_N}^2 \leq u_N^\top \tilde G_N u_N=-u_N^\top F^\prime(x_N)\leq \norm[]{u_N}\norm[]{F^\prime(x_N)}\,.
\end{equation}
Thus, $ \norm[]{u_N}\leq \frac{\norm[]{F^\prime(x_N)}}{\mu\lambda_N}$. 
Therefore,
\begin{equation}
    L_Hr_N= L_H\norm[]{u_N}\leq \frac{L_H\norm[]{F^\prime(x_N)}}{\mu\lambda_N}\leq \sqrt{L_H\norm[]{F^\prime(x_N)} }\,,
\end{equation}
where the last inequality holds since $\mu\lambda_N\geq \sqrt{L_H\norm[]{F^\prime(x_N)} }$ (see the definition of $\lambda_k$).
We can deduce that:
\begin{equation}
     L_Hr_N + L_Hr_{N-1} \leq \sqrt{L_H\norm[]{F^\prime(x_N)}} + L_Hr_{N-1}=\mu\lambda_N\,.
\end{equation}
Therefore, we deduce by the induction hypothesis and Lemma~\ref{lemma:JleqJ} that
\begin{equation}
    \begin{split}
         \ \tilde G_N &= G_N(1+\lambda_N)\\ 
        &\succeq J_\Nm(1 + \lambda_N)\\&\succeq J_\Nm + \mu\lambda_N\opid\\&\succeq J_\Nm + (L_Hr_\Nm + L_Hr_N)\opid\\&\succeq J_N\,.\\
    \end{split}
    \end{equation}
    Using Lemma~\ref{lemma:Anton} again, we obtain $J_N\preceq G_\Np=\text{SR1}(J_N,\tilde G_N,u_N) \preceq \tilde G_N$ and $\trace G_\Np\leq \trace \tilde G_N\leq n\bar\kappa$. 
    
\item If $\trace \hat G_N>n\bar \kappa$, then $\tilde G_N=L\opid$ due to the restarting step.
    Automatically, we have $\tilde G_N\succeq J_N$ and using Lemma~\ref{lemma:Anton} again, we obtain $J_N\preceq G_{N+1}\preceq \tilde G_N\preceq L\opid$ and $\trace G_\Np\leq \trace \tilde G_N=nL\leq n\bar\kappa$.\qedhere
    \end{enumerate}
    % Then we discuss by cases to prove the rest two results:
    % \begin{enumerate}[1.]
    %     \item If $\trace \hat G_{N}\leq n\bar\kappa$, then, $\tilde G_N=\hat G_N$. 
    % %     According to Lemma~\ref{lemma:Anton} and $G_N=SR(J_{N-1},\tilde G_{N-1},u_{N-1})$, we obtain that 
    % % \begin{equation}
    % %     \tilde G_{N}=(1+\lambda_N)G_N\preceq (1+\lambda_N)\tilde G_{N-1}\preceq \left(\prod_{i=0}^{N}(1+\lambda_i)\right)L\opid\,.
    % % \end{equation}
    % % Thus, we obtain
    % \begin{equation}
    % \begin{split}
    %      (\text{Induction hypothesis})\ \tilde G_N &\succeq J_\Nm+\lambda_N\\ 
    %     &\succeq J_\Nm + 2L_Hr_\Nm + 2L_Hr_N\\(\text{Lemma~\ref{lemma:JleqJ}})&\succeq J_N\,.\\
    % \end{split}
    % \end{equation}
    % Using Lemma~\ref{lemma:Anton} again, we obtain $J_N\preceq G_\Np=SR(J_N,\tilde G_N,u_N) \preceq \tilde G_N$.
    % \item If $\trace \hat G_N>n\bar \kappa$, then $\tilde G_N=L\opid$.
    % Automatically, we have $\tilde G_N\succeq J_N$ and using Lemma~\ref{lemma:Anton} again, we obtain $G_{N+1}\succeq J_N$.
    % \end{enumerate}
\end{proof}
\begin{remark}
    Lemma~\ref{lemma2:bound} indicates that both $G_\kp$ and $\tilde G_k$ are bounded by $n\bar\kappa\opid$ for any $k\in\N$ due to the restarting strategy, while the candidate $\hat G_k$ has the possibility to be larger than $n\bar\kappa\opid$ where $\bar\kappa\geq L$. Since our potential function is simple and does not involve any Hessian term, the computational cost of the simple restarting criterion $\trace \hat G_\kp\leq n\bar\kappa$ is only $O(n)$, which is another advantage of our potential function. 
\end{remark}

Based on the lemma above, we can derive a descent lemma for the objective values.
\begin{lemma}[descent lemma for function values]\label{lemma2:descentVal}
For any $k\in\N$, we have
\begin{equation}
    F(x_\kp) - F(x_k)\leq-\frac{\mu}{2}r_k^2\,.
\end{equation}
\end{lemma}
\begin{proof} The argument is similar with the one of Lemma~\ref{lemma:fdescent}. By the strong convexity of $f$, we have
    \begin{equation}\label{grad decentlemma ineq:1}
        \begin{split}
            f(x_\kp)-f(x_k)&\leq \scal{\nabla f(x_k)}{x_\kp-x_k}+\scal{\nabla f(x_\kp)-\nabla f(x_k)}{x_\kp-x_k} - \frac{\mu}{2}\norm[]{x_\kp-x_k}^2\\
            &=\scal{\nabla f(x_k)}{x_\kp-x_k}+\scal{J_k(x_\kp-x_k)}{x_\kp-x_k} - \frac{\mu}{2}\norm[]{x_\kp-x_k}^2\\
            % \text{Lemma~\ref{lemma:JleqJ}}&\leq \scal{\nabla f(x_k)}{x_\kp-x_k}+\scal{J_\km(x_\kp-x_k)}{x_\kp-x_k}
            % +L_H(r_k+r_\km)\norm[]{x_\kp-x_k}^2\\
            % & - \frac{\mu}{2}\norm[]{x_\kp-x_k}^2\\
            &\leq \scal{\nabla f(x_k)}{x_\kp-x_k}+\scal{\tilde G_k(x_\kp-x_k)}{x_\kp-x_k}
            - \frac{\mu}{2}\norm[]{x_\kp-x_k}^2\,,\\
            % \text{equation~\ref{optimalitycond}}&=-L_H(r_k+r_\km)\norm[]{x_\kp-x_k}^2-\frac{\mu}{2}r_k^2\,\\
        \end{split}
    \end{equation}
    where the second inequality above holds due to Lemma~\ref{lemma2:bound}.
    Since $x_\kp$ is the solution of \eqref{optimality condition:grad} and $g(x)$ is convex, we can derive that 
    \begin{equation}\label{grad decentlemma ineq:2}
        g(x_\kp) -g(x_k)\leq \scal{v_\kp}{x_\kp-x_k}\,.
    \end{equation}
    Combining \eqref{grad decentlemma ineq:1} and \eqref{grad decentlemma ineq:2}, we deduce that 
    \begin{equation}
    \begin{split}
        F(x_\kp) -F(x_k)&\leq \scal{v_\kp+ \nabla f(x_k)}{x_\kp-x_k}+\scal{\tilde G_k(x_\kp-x_k)}{x_\kp-x_k}
            - \frac{\mu}{2}\norm[]{x_\kp-x_k}^2\\
            &=-\frac{\mu}{2}r_k^2\,,\\
    \end{split}
    \end{equation}
    where the last equality holds due to \eqref{optimality condition:grad}.
\end{proof}
% Due to the descent lemma for function values, we can have a bound on $\norm[]{F^\prime(x_k)}$.
% \begin{lemma}\label{lemma2:lambdakbound}
%     There exists some $D>0$, such that $\lambda_k\leq D$ for any $k\in\N$.
% \end{lemma}
% \begin{proof}
%     By Lemma~\ref{lemma2:descentVal}, we have for any $k\in N$, $x_k\in\set{x\vert f(x)\leq f(x_0)}$. Due to the strong convexity, $\set{x\vert f(x)\leq f(x_0)}$ is bounded. Thus, $\norm[]{\nabla f(x_k)}$ and $r_k$ are bounded for each $k\in\N$. Besides, $\tilde G_k$ is bounded by $n\bar \kappa$. Thus, $F^\prime(x_k)=\nabla f(x_k)-\nabla f(x_\km)-\tilde G_\km(x_k-x_\km)$ is bounded. Therefore, we can conclude that there exists some $D>0$ such that 
%     $$\lambda_k=2L_Hr_\km+\sqrt{2L_H\norm[]{F^\prime(x_k)}}\leq D\,.$$
% \end{proof}
Next, we study the growth rates of several important sequences.
\begin{lemma}\label{lemma2:growthrk}
Given a number of iteration $N\in\N$, we have:
    \begin{gather}
        \sum_{k=0}^{\textcolor{blue}{N}} r_k^2\leq \frac{2(F(x_0)-\inf F)}{\mu} < \infty \,; \label{eq:lemma2:growthrk-1}\\
        \sum_{k=0}^{N-1}r_k\leq C_0 N^{1/2}\quad \text{and}\quad \sum_{k=1}^{N}r_k\leq C_0 N^{1/2} \,; \\
        \sum_{k=1}^N \lambda_k\leq \Theta N^{3/4}\quad \text{where}\ \Theta := \frac{1}{\mu}(L_HC_0 +\sqrt{L_H(L+n\bar\kappa)C_0})\,,
    \end{gather} 
    %%\begin{enumerate}
    %%    \item $\sum_{k=0}^{\textcolor{blue}{N}} r_k^2\leq \frac{2(F(x_0)-\inf F)}{\mu}$;
    %%    \item $\sum_{k=0}^{N-1}r_k\leq C_0 N^{1/2}$ and $\sum_{k=1}^{N}r_k\leq C_0 N^{1/2}$ ;
    %%    \item $\sum_{k=1}^N \lambda_k\leq \Theta N^{3/4}$ where $\Theta := \frac{1}{\mu}(L_HC_0 +\sqrt{L_H(L+n\bar\kappa)C_0})$ is a constant.
    %%\end{enumerate}
    where $C_0\coloneqq \sqrt{\frac{2(F(x_0)-\inf F)}{\mu}}$.
\end{lemma}
\begin{proof}
    We start with the first and the second results.
    By summing the inequality in Lemma~\ref{lemma2:descentVal} from $k=0$ to $N-1$, we have
    \begin{equation}
        \frac{\mu}{2}\sum_{k=0}^{N-1} r_k^2\leq \sum_{k=0}^{N-1}F(x_k)-F(x_\kp)\leq F(x_0)-\inf F\,,
    \end{equation}
    which yields the first result. Using Cauchy--Schwarz inequality, we obtain:
     \begin{equation}
        \sum_{k=0}^{N-1} r_k\leq \left(\sum_{k=0}^{N-1} r_k^2\right)^{1/2}\left(\sum_{k=0}^{N-1} 1\right)^{1/2}\leq C_0 N^{1/2}\,,
    \end{equation}
    and because of the uniform bound above, we conclude that
    \begin{equation}
        \sum_{k=1}^{N} r_k\leq\left(\sum_{k=1}^{N} r_k^2\right)^{1/2}\left(\sum_{k=1}^{N} 1\right)^{1/2}\leq C_0 N^{1/2}\,.
    \end{equation}
    Now, we prove the third result.
    From \eqref{optimality condition:grad} we derive that:
    \begin{equation}\label{ineq:ak}
    \begin{split}
         L_H\norm[]{F^\prime (x_k)} &= L_H\norm[]{\nabla f(x_k)-\nabla f(x_\km)-\tilde G_\km(x_k-x_\km)}\\&\leq L_H(\norm[]{\nabla f(x_k)-\nabla f(x_\km)} +\norm[]{\tilde G_\km(x_k-x_\km)  })\\&\leq L_H(L+n\bar\kappa)r_\km\,,\\
    \end{split}
    \end{equation}
    where the last inequality holds since $\nabla f$ is $L$-Lipschitz continuous and $\tilde G_k$ is bounded uniformly for any $k\in\N$ by the restarting step in Algorithm~\ref{Alg:SR1_da_PQN2}. For convenience, we denote $a_k:=\sqrt{L_H\norm[]{F^\prime (x_k)} }$ for all $k\in\N$. From \eqref{ineq:ak} and \eqref{eq:lemma2:growthrk-1}, we get
    \begin{equation}
        \sum_{k=1}^N a_k^2\leq L_H(L+n\bar\kappa)\sum_{k=0}^{N-1} \textcolor{blue}{r_k}
       \leq L_H(L+n\bar\kappa)\textcolor{blue}{C_0}N^{1/2} \,.
    \end{equation}
    Then, by Cauchy--Schwarz inequality, we obtain
    \begin{equation}
        \sum_{k=1}^{N}a_k\leq\left(\sum_{k=1}^{N}a_k^2\right)^{1/2}\left(\sum_{k=1}^N 1\right)^{1/2}\leq \sqrt{L_H(L+n\bar\kappa)C_0} N^{1/4}*N^{1/2}\leq \sqrt{L_H(L+n\bar\kappa)C_0} N^{3/4}\,.
    \end{equation}
    Thus, we have the growth rate:
    \begin{equation}
        \sum_{k=1}^N \lambda_k = \frac{1}{\mu}\left(\sum_{k=1}^N L_H r_\km + \sum_{k=1}^{N}a_k\right)\leq \frac{1}{\mu}\left(L_HC_0N^{1/2} + \sqrt{L_H(L+n\bar\kappa)C_0} N^{3/4}\right)\leq\Theta N^{3/4}\,,
    \end{equation}
    where $\Theta = \frac{1}{\mu}\left(L_HC_0 +\sqrt{L_H(L+n\bar\kappa)C_0}\right)$. 
\end{proof}
% We can reuse Lemma~\ref{lemma:JJGG}, the second result of Lemma~\ref{lemma:growth} and Lemma~\ref{lemma:potentialfunction} because they still hold true with the new setting of $\lambda_k$. Then we can derive similar descent lemma as in cubic quasi-Newton method by slightly adapting the proof.
\begin{lemma}\label{grad:ukGu}
For any $k\in \N$, we have 
\begin{equation}
    u_k^\top \tilde G_k u_k\leq \frac{1}{\mu}\norm[]{F^\prime (x_k)}^2\,.
\end{equation}
\end{lemma}
\begin{proof}
% Since $g$ is convex, we have $u_k^\top(v_\kp-v_k)\geq 0$. 
% Then, there exists a symmetric positive semidefinite matrix $B_k\coloneqq \frac{(v_\kp-v_k)(v_\kp-v_k)^\top}{ u_k^\top(v_\kp-v_k)} $ such that $B_ku_k =(v_\kp-v_k)$.
According to \eqref{optimality condition:grad}, we obtain that $\tilde G_ku_k=-\nabla f(x_k)-v_\kp$.
% By the fact $\tilde G_k\succeq \mu\opid$ (Lemma~\ref{lemma2:bound}), we can deduce that:
%     \begin{equation}
%     \begin{split}
%          u_k^\top \tilde G_k u_k&\leq u_k^\top \tilde G_k u_k+ u_k^\top (v_\kp-v_k)\\&=u_k^\top( -\nabla f(x_k)-v_\kp+v_\kp-v_k   )\\&= u_k^\top F^\prime(x_k)\\&\leq  \frac{1}{2}u_k^\top \tilde G_k u_k +\frac{1}{2}\norm[\tilde G_k^{-1}]{F^\prime(x_k)}^2\\
%          &\leq  \frac{1}{2}u_k^\top \tilde G_k u_k +\frac{1}{2\mu}\norm[]{F^\prime(x_k)}^2\,,
%     \end{split}
%     \end{equation}
%     where the last inequality holds since  $\tilde G_k^{-1}\preceq \frac{1}{\mu}\opid$. By rearranging the above inequality, 
The rest of this proof remains the same as the one for Lemma~\ref{ccubic:ukGu}.
\end{proof}
\begin{lemma}\label{lemma2:maindescent}
For any $k\in\N$, we have a descent inequality as the following:
\begin{equation}
    V(\tilde G_k)-V(\tilde G_\kp)\geq \frac{\mu g^2_\kp}{g_k^2}-n\lambda_\kp\bar\kappa\,,
\end{equation}
where $g_k\coloneqq \norm[]{F^\prime(x_k)}$.
\end{lemma}
\begin{proof}
By the optimality condition in \eqref{optimality condition:grad}, we have $\tilde G_k u_k+v_\kp= -\nabla f(x_k)$ and by definition of $J_k$, we have $J_ku_k=\nabla f(x_\kp)-\nabla f(x_k)$. 
Using Lemma~\ref{lemma:potentialfunction}, we have the following estimation:
    \begin{equation}\label{ineq2:V3}
    \begin{split}
         \textcolor{black}{V(\tilde G_k)-V(G_\kp)=\nu(J_k,\tilde G_k,u_k)} &\geq \frac{\norm[]{\tilde G_k-J_k)u_k}^2}{u_k^\top \tilde G_k u_k} = \frac{\norm[]{\nabla f(x_\kp)+v_\kp}^2}{u_k^\top \tilde G_k u_k}\\
         &= \frac{\norm[]{F^\prime (x_\kp)}^2}{u_k^\top\tilde G_k u_k}\geq \frac{\mu\norm[]{F^\prime (x_\kp)}^2}{\norm[]{F^\prime(x_k)}^2}=\frac{\mu g_\kp^2}{g_k^2}\,,
    \end{split}
    \end{equation}
    where the second inequality holds due to \text{Lemma~\ref{grad:ukGu}}. 
    
    % Together with Lemma~\ref{lemma:potentialfunction} on the $k$-th iterate, we obtain
    % \begin{equation}\label{ineq2:V3}
    %     V(\tilde G_k) - V(G_\kp)=\textcolor{black}{\nu(J_k,\tilde G_k,u_k)}\geq\frac{\mu g_\kp^2}{g_k^2} \,.
    % \end{equation}
    We also need a lower bound for $V(G_\kp)-V(\tilde G_\kp) $, for which we need to take the two cases of the restarting Step~\ref{alg:SR1_da_PQN2-step4} into account.
    \begin{enumerate}[1.]
        \item When $\trace\hat G_\kp\leq n\bar\kappa$, we have $\tilde G_\kp=\hat G_\kp$, $\trace G_\kp\leq n\bar\kappa$ and
        \begin{equation}\label{eq:lemma2:maindescent:proof1}
        V(G_\kp)-V(\tilde G_\kp) = \trace(G_\kp-G_\kp-\lambda_\kp G_\kp) \geq - n\lambda_\kp\bar\kappa\,.
    \end{equation}
    \item When $\trace\hat G_\kp> n\bar\kappa$, we have $V(\hat G_\kp)>n\bar\kappa\geq nL=\trace L \opid$. Since in this case we set $\tilde G_\kp=L\opid$, we have $ V(\hat G_\kp)\geq V(\tilde G_\kp)$. According to Lemma~\ref{lemma2:bound}, we have $\trace G_\kp\leq n\bar\kappa$. Since $\lambda_k\geq0$, we deduce that
    \begin{equation}\label{eq:lemma2:maindescent:proof2}
    \begin{split}
        V(G_\kp)-V(\tilde G_\kp)&= V(G_\kp)-V(\hat G_\kp) +V(\hat G_\kp)-V(\tilde G_\kp)\\&\geq V(G_\kp)-V(\hat G_\kp)  \\
       &\geq \trace(G_\kp-G_\kp-\lambda_\kp G_\kp)\\
       &\geq -\lambda_\kp\trace(G_\kp)\\
        &\geq - n\lambda_\kp\bar\kappa\,.\\
    \end{split}
    \end{equation}
    \end{enumerate}
    Therefore, adding \eqref{ineq2:V3} with either \eqref{eq:lemma2:maindescent:proof1} or \eqref{eq:lemma2:maindescent:proof2} yields the desired inequality.
\end{proof}

Now, we are ready to prove our main theorem for Algorithm~\ref{Alg:SR1_da_PQN2}. % \begin{theorem}[Main Theorem 2]\label{Maintheorem2}
% For any initialization $x_0\in\R^n$ and any $N\in\N$, grad SR1 PQN has a global convergence with the rate:
% \begin{equation}\label{ineq2:super-linearrate}
%     \norm[]{F^\prime (x_N)}\leq \left(\frac{C_{\mathrm{grad}}}{N^{1/4}}\right)^{N/2}\norm[]{F^\prime (x_0)} \,,
% \end{equation}
% where $C_{\mathrm{grad}}\coloneqq (n\bar\kappa\Theta + V(G_0))/\mu$, $C_0\coloneqq \sqrt{\frac{ 2(F(x_0) -\inf F)}{\mu}}$ and $\Theta = \frac{1}{\mu}\left(L_HC_0 +\sqrt{L_H(L+n\bar\kappa)C_0}\right)$.
% \end{theorem}
\begin{proof}[\textbf{Proof of Theorem~\ref{thm:main-grad}}]
For symplicity, we denote $\gamma_k^2 \coloneqq\frac{g^2_\kp}{g^2_k}$.
    Summing the result in Lemma~\ref{lemma2:maindescent} from $k=0$ to $N-1$, we obtain 
    \begin{equation}
        V(\tilde G_0)-V(\tilde G_N)\geq \mu\sum_{k=0}^{N-1} \gamma_k^2 - n\bar\kappa\sum_{k=1}^{N} \lambda_k\,.
    \end{equation}
    Since, for every $k$, our method keeps $\tilde G_k\succeq J_k$ and $G_\kp\succeq J_k$ (see Lemma~\ref{lemma2:bound}),  we have $V(\tilde G_k)>0$ and therefore
    \begin{equation}\label{ineq2:sumineq}
    \begin{split}
    V(\tilde G_0)&\geq \mu \sum_{k=0}^{N-1} \gamma_k^2 - n\bar\kappa\sum_{k=1}^{N} \lambda_k\,.\\
    \end{split}
    \end{equation}
    By Lemma~\ref{lemma2:growthrk}, we obtain
    \begin{equation}\label{mainstep2}
    C_{\mathrm{grad}}N^{3/4}\geq \frac{1}{\mu}\left(V(G_0)+ n\bar\kappa\Theta N^{3/4}\right)\geq \sum_{k=0}^{N-1}\gamma_k^2\,,
    \end{equation}
    where $C_{\mathrm{grad}}\coloneqq (n\bar\kappa\Theta + V(G_0))/\mu$. Dividing  \eqref{mainstep2} by $N$, we obtain 
    \begin{equation}\label{mainstep2.1}
    \frac{C_{\mathrm{grad}}}{N^{1/4}}\geq \frac{1}{N}\sum_{k=0}^{N-1}\gamma_k^2\,.
    \end{equation}
    We derive from \eqref{mainstep2.1} with the concavity and monotonicity of $\log x$ that
    \begin{equation}
    \begin{split}
          \log\Big(\frac{C_{\mathrm{grad}}}{N^{1/4}}\Big) &\geq \log\left(\frac{1}{N}\sum_{k=0}^{N-1}\gamma_k^2\right)\geq\frac{1}{N}\sum_{k=0}^{N-1}\log(\gamma_k^2)= \log \left(\left(\prod_{k=0}^{N-1}\frac{g_\kp^2}{g_k^2}\right)^{1/N}\right)=\log\left(\left(\frac{g_N}{g_0}\right)^{2/N}\right)\,.\\
    \end{split}
    \end{equation}
    Thus, we obtain 
    \begin{equation}
        g_N\leq \left(\frac{C_{\mathrm{grad}}}{N^{1/4}}\right)^{N/2}g_0\,.\qedhere
    \end{equation}
    
\end{proof}

\section{Experiments}

In the following section, we consider two regression applications from machine learning to provide also numerical evidence about the superior performance of our algorithms and thereby validating the global non-asymptotic super-linear rates of convergence. While our algorithms can solve possibly non-smooth additive composite optimization problems, we restrict ourselves to smooth problems, for which our algorithms are new as well. This is mainly due to the fact that the efficient solution of sub-problems in the non-smooth case needs additional careful investigations, possibly along the lines of \cite{becker2019quasi}, which we will tackle in future work. Our code is available on Github: \url{https://github.com/wsdxiaohao/NonAsymSuConvSR1QN.git}

\subsection{Regularized log-sum-exp problem}

For experiment, we test our methods: \textup{Cubic SR1 PQN}(Algorithm \ref{Alg:SR1_da_PQN}) and \textup{Grad SR1 PQN} (Algorithm \ref{Alg:SR1_da_PQN2}) on the log-sum-exp problem with  regularization:
\begin{equation}
    \min_{x\in\R^n} f(x)\,,\quad f(x) = \log\Big(\sum_{i=1}^m \exp{(\scal{a_i}{x}-b_i)} \Big) +\frac{\mu}{2}\norm[]{x}^2\,,
\end{equation}
where $a_1, a_2, \cdots, a_m\in\R^n$, $b_1, b_2, \cdots, b_m\in\R$ are given. The strong convexity modulus is $\mu=1$. The Lipschitz constant is $L=\mu+2\sum_{i=1}^m\norm[]{a_i}^2$ and $L_H=2$. In our experiment, $(a_i)_{i=1,\ldots, m}\subset\R^n$ and $(b_i)_{i=1,\ldots, m}\subset \R$ are generated randomly with $m=500$ and $n=200$. We set $\bar\kappa =\textcolor{blue}{3L}$ for \textup{(Grad SR1 PQN)}. As depicted in Figure~\ref{fig:log-sum-exp}, our methods achieve a super-linear rate of convergence. However, since cubic regularization is hard to compute, Cubic Newton \cite{nesterov2006cubic} and \textup{Cubic SR1 PQN} are not competitive when it comes to measuring the actual computation time. This is in contrast to the Heavy ball Method (HBF) \cite{PolyakBook} (with optimal damping rate $\beta =\frac{\sqrt{L}-\sqrt{\mu}}{\sqrt{L}+\sqrt{\mu}}$ and step size $\tau=\frac{4}{(\sqrt{L}+\sqrt{\mu})^2}$) and Gradient Descent (GD), which both have cheap cost per iteration. Our gradient regularized quasi-Newton method outperforms all other methods significantly, especially, in initial iterations. However, the estimation of convergence rates of our method is not sharp, despite our methods perform much better in practice, which is the price to pay for globality. 
%\textcolor{red}{\st{while the starting point is far from the optimal solution,}}
\begin{figure}[htp]
    \centering
    \includegraphics[width=0.48\linewidth]{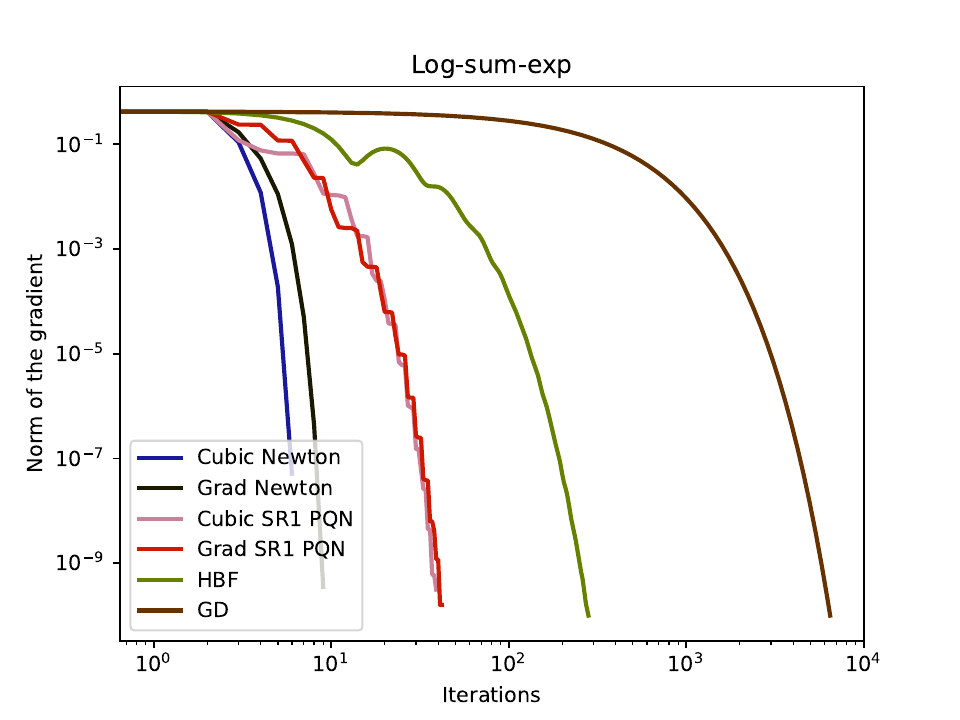}
    \includegraphics[width=0.48\linewidth]{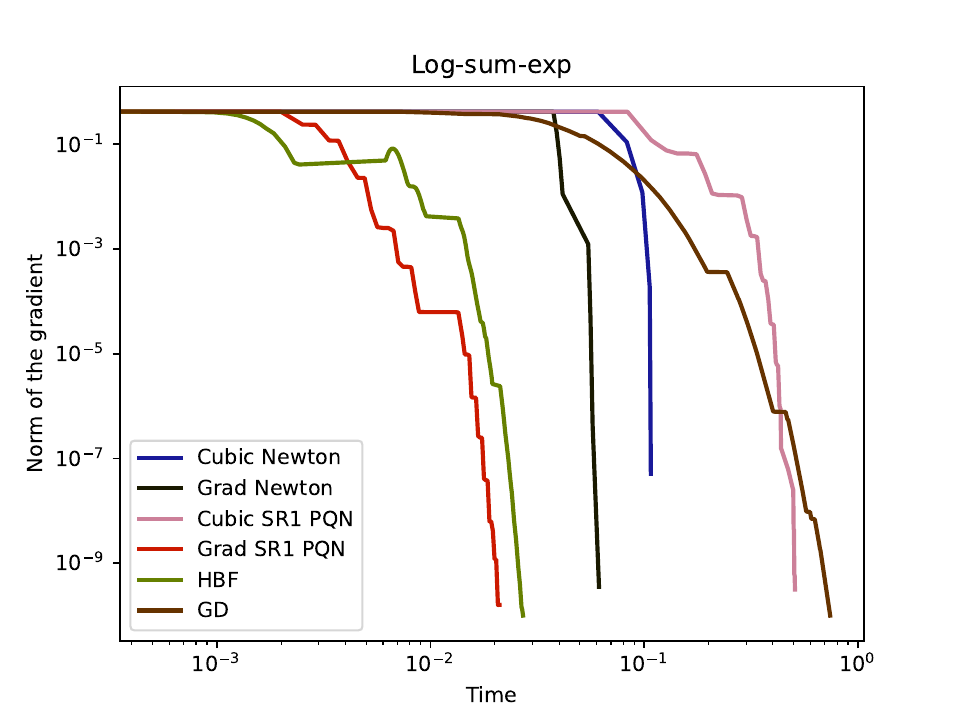}
    \caption{ \textcolor{black}{Our proposed \textup{Cubic SR1 PQN} and \textup{Grad SR1 PQN} significantly outperform first-order methods, despite slower than related Newton methods. In terms of time, our \textup{Grad SR1 PQN} is the fastest, while methods with cubic regularization take longer time when the initial point is far from the optimal point.
}}
    \label{fig:log-sum-exp}
    %\label{fig:reg-SR1}
\end{figure}

\subsection{Logistic regression}
The second experiment is a logistic regression problem with regularization on the benchmark dataset ``mushroom'' from UCI Machine Learning Repository \cite{mushroom_73}:
\begin{equation}
    \min_{x\in\R^n} f(x),\quad \text{with}\ f(x)\coloneqq \frac{1}{m}\sum_{i=1}^m\log(1+\exp(-b_ia_i^\top x)) + \frac{\mu}{2}\norm[]{x}^2\,,
\end{equation}
where $a_i\in\R^n$, $b_i\in\R$ for $i=1,2,\cdots,m$ denotes the given data from the ``mushrooms'' dataset. The strong convexity modulus is $\mu=0.1$. The Lipschitz constants are $L=\mu+2\sum_{i=1}^m\norm[]{a_i}^2$ and $L_H=2$.  Here, $m=8124$ and $n=117$. We set $\bar\kappa = \textcolor{black}{4L}$ for \textup{(Grad SR1 PQN)}.
\begin{figure}[H]
    \centering
    \includegraphics[width=0.48\linewidth]{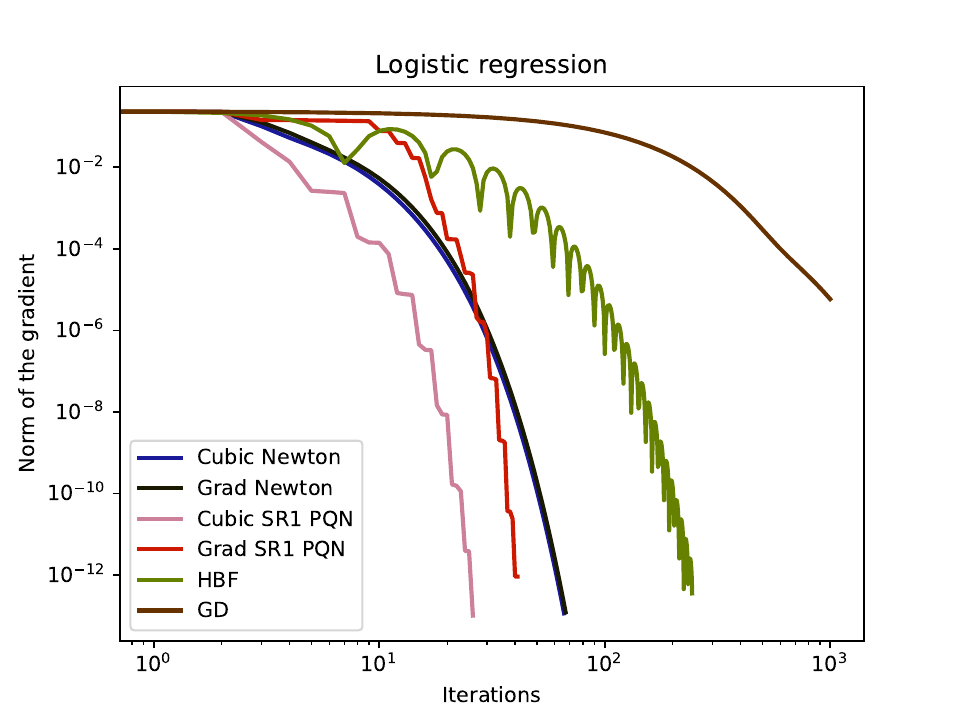}
    \includegraphics[width=0.48\linewidth]{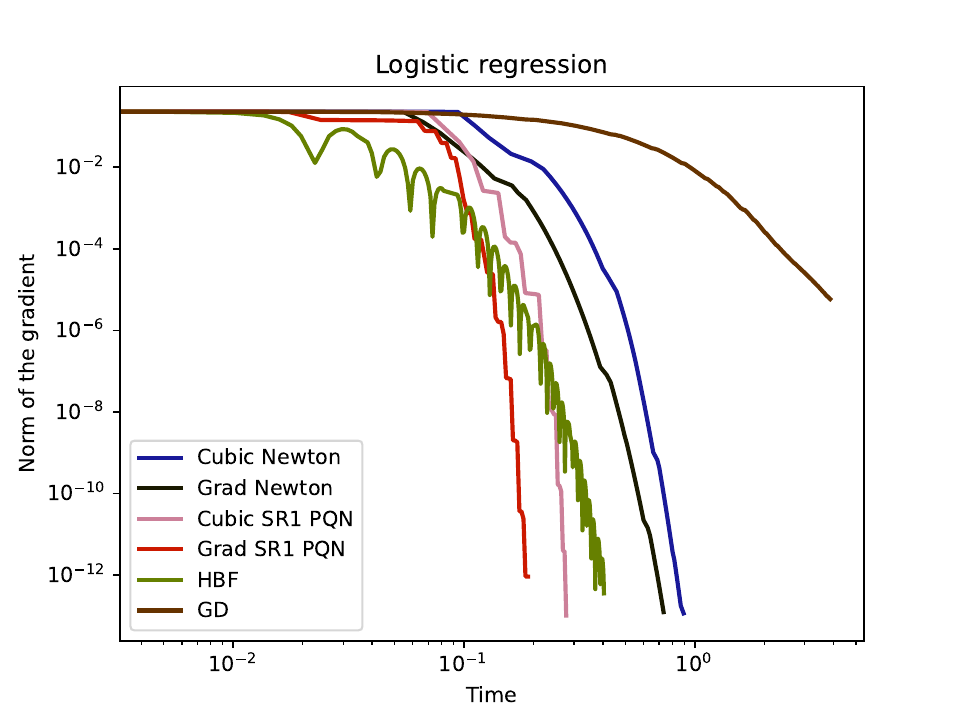}
    \caption{Our proposed \textup{Cubic- and Grad SR1 PQN} still significantly outperforms first-order methods in number of iterations. In terms of time, our \textup{Grad SR1 PQN} is \textcolor{black}{eventually} among the fastest methods due to SR1 metric and the cheap computation of the correction step and the restarting step.}
    \label{fig:reg-SR1-logistic}
\end{figure}
As depicted in Figure~\ref{fig:reg-SR1-logistic}, our methods achieve a super-linear rate of convergence globally. Since the computation for the quasi-Newton metric is less than that of Newton's method, Cubic quasi-Newton is better than Cubic Newton eventually in terms of time. However, due to the cubic term, neither Cubic Newton nor Cubic quasi-Newton is competitive when it comes to measuring the actual computation time. Our gradient regularized SR1 quasi-Newton method outperforms all other methods significantly in number of iterations and is among the fastest in terms of time.

\section{Conclusion}

 In this paper, we propose two variants of proximal quasi-Newton SR1 methods which converge globally with an explicit (non-asymptotic) super-linear rate of convergence. The key is the adaptation of cubic and gradient regularization from the pioneering works \cite{nesterov2006cubic,mishchenko2023regularized} on Newton's method to quasi-Newton methods. The potential function in our paper differs from previous works \cite{rodomanov2022rates,jin2023non,ye2023towards} and turns out to be more tractable for the analysis of the regularized SR1 quasi-Newton method, also by making the restarting step numerically easier. Moreover, our algorithm and analysis is directly developed for non-smooth additive composite problems. The analysis of the convergence used in this paper paves the road to generalize many regularized Newton methods, in particular, \cite{mishchenko2023regularized,doikov2024gradient,doikov2024super}. Since this paper considers the SR1 method, an interesting generalization is that of proximal quasi-Newton methods with BFGS or DFP metric update under the same conditions, especially without using line search.

\appendix
\section{A variant of gradient regularized proximal quasi-Newton method} \label{sec:grad-reg-second-alg}
\textbf{Algorithm~\ref{Alg:SR1_da_PQN3}} is almost the same with Algorithm~\ref{Alg:SR1_da_PQN2}. Step~\ref{alg:SR1_da_PQN3-step1} of Algorithm~\ref{Alg:SR1_da_PQN3} equips the classical quasi-Newton term with an additive gradient regularization term instead of a scaling factor involving gradient. Step~\ref{alg:SR1_da_PQN3-step2} defines $J_k$ as in Algorithm~\ref{Alg:SR1_da_PQN2}. In Step~\ref{alg:SR1_da_PQN3-step3}, several variables are defined, including $F^\prime(x_\kp)$ which will be used to generate the regularization term, compute $\lambda_k$ and terminate the algorithm in Step~\ref{alg:SR1_da_PQN3-step5}. Step~\ref{alg:SR1_da_PQN3-step4} computes the corrected metric $\hat G_\kp$ by adding $\lambda_k$. As in Algorithm~\ref{Alg:SR1_da_PQN2}, we adopt the same restarting strategy for $\tilde G_\kp$ in Algorithm~\ref{Alg:SR1_da_PQN3}.

\begin{algorithm}[H]
\caption{Grad Reg SR1 PQN}\label{Alg:SR1_da_PQN3}
\begin{algorithmic}
\Require $x_0$, $G_0=L\opid$, $\lambda_0=0$, $\tilde G_0=G_0+\lambda_0$, $\bar \kappa\geq L$.
\State \textbf{Update for $k = 0,\cdots,N$ }:
\begin{enumerate}[1.]
\item \label{alg:SR1_da_PQN3-step1} \textbf{Update} \begin{equation}\label{alg3:update}
    x_\kp =\argmin_{x\in\R^n} g(x) + \scal{\nabla f(x_k)}{x-x_k}+\frac{1}{2}\norm[\tilde G_k]{x-x_k}^2\,.
\end{equation}
    \item \label{alg:SR1_da_PQN3-step2}\textbf{Denote but not use explicitly} $J_k\coloneqq\int_0^1\nabla^2f(x_k+tu_k)dt$.
    \item \label{alg:SR1_da_PQN3-step3}\textbf{Compute} $u_k=x_\kp-x_k$, $r_k=\norm[]{u_k}$ and 
    \[
    \begin{split}
    F^\prime(x_\kp)&=\nabla f(x_\kp)-\nabla f(x_k)-\tilde G_{k}u_k\\
    G_\kp&=\text{SR1}(J_k,\tilde G_k,u_k)\\
    \lambda_\kp&=\left(\sqrt{L_H\norm[]{F^\prime (x_\kp)}}+L_Hr_{k}\right)\,.\\
    \end{split}
    \]
    
    \item \label{alg:SR1_da_PQN3-step4}\textbf{Update $\tilde G_\kp$:} compute $\hat G_\kp=G_\kp+\lambda_\kp$ (\textbf{Correction step}).
    \begin{enumerate}[]
        \item \textbf{If $\trace  \hat G_\kp\leq n\bar\kappa$:} we set $\tilde G_\kp= \hat G_\kp$.
        \item \textbf{Othewise:} we set $\tilde G_\kp= L\opid$ (\textbf{Restart step}).
    \end{enumerate}
    \item \label{alg:SR1_da_PQN3-step5}\textbf{If $\norm{F^\prime(x_\kp)}=0$}:
    Terminate.
\end{enumerate}
\State \textbf{End}
\end{algorithmic}
\end{algorithm}

In Algorithm~\ref{Alg:SR1_da_PQN3}, we have $\tilde G_k=G_k+\lambda_k$ where $G_k$ is generated by the SR1 method or $\tilde G_k=L\opid$. We consider the smooth case that $g=0$. When $ \tilde G_k=G_k+\lambda_k$, we can not use Sherman-Morrison-Woodbury formula to derive the inverse of $\tilde G_k$ due to the addition. Thus, even for smooth case, the computational cost of update step is the same as the inversion of a matrix, namely, $O(n^3)$. However, for non-smooth setting, the computational cost of the update step in Algorithm~\ref{Alg:SR1_da_PQN3} is the same as that of classical proximal Newton method.
For Algorithm~\ref{Alg:SR1_da_PQN3}, we have the following non-asymptotic super-linear rate.

\begin{theorem}[Main Theorem of Algorithm~\ref{Alg:SR1_da_PQN3}]\label{thm:main-gradreg}
For any initialization $x_0\in\R^n$ and any $N\in\N$, \textup{(Grad Reg SR1 PQN)} has a global convergence with the rate:
\begin{equation}\label{ineq3:super-linearrate}
    \norm[]{F^\prime (x_N)}\leq \left(\frac{C_{\mathrm{grad}}}{N^{1/4}}\right)^{N/2}\norm[]{F^\prime (x_0)} \,,
\end{equation}
where $C_{\mathrm{grad}}\coloneqq (\Theta + nL)/\mu$, $C_0\coloneqq \sqrt{\frac{ 2(F(x_0) -\inf F)}{\mu}}$ and $\Theta = L_HC_0 +\sqrt{L_H(L+n\bar\kappa)C_0}$.
\end{theorem}
\begin{proof}
    See Section~\ref{conv:alg3}.
\end{proof}

\begin{remark}
    In fact, \eqref{ineq3:super-linearrate} indicates that the super-linear rate is only attained after $N\geq C_{\mathrm{grad}}^4$ number of iterations, where $C_{\mathrm{grad}}=O(\frac{n}{\mu^{3/2}})$. 
\end{remark}

\section{Convergence analysis of Algorithm~\ref{Alg:SR1_da_PQN3}}\label{conv:alg3}
Let $\lambda_k\coloneqq \left(L_Hr_\km+\sqrt{L_H\norm[]{F^\prime (x_k)}}\right)$. We obtain a gradient regularized SR1 proximal quasi-Newton method which can be also regarded as a generalization of gradient regularized proximal Newton method \cite{doikov2024gradient}.
The convergence analysis is very similar with the one for Algorithm~\ref{Alg:SR1_da_PQN2} since the only difference in Algorithm~\ref{Alg:SR1_da_PQN3} is that $\tilde G_\kp= G_\kp +\lambda_\kp$ with $\lambda_\kp= \left(2L_Hr_\km+\sqrt{2L_H\norm[]{F^\prime (x_k)}}\right)$. In this case, we have $\sum_{k=1}^N \lambda_k\leq \Theta N^{3/4}$ where $\Theta = L_HC_0 +\sqrt{L_H(L+n\bar\kappa)C_0}$ is a constant. The proof of that property is very similar with the one of Lemma~\ref{lemma2:growthrk}. Most properties for Algorithm~\ref{Alg:SR1_da_PQN2} listed in previous section still hold true for Algorithm~\ref{Alg:SR1_da_PQN3} except the following lemma.
\begin{lemma}\label{lemma3:maindescent}
For any $k\in\N$, we have a descent inequality as the following:
\begin{equation}
    V(\tilde G_k)-V(\tilde G_\kp)\geq \frac{\mu g^2_\kp}{g_k^2}-n\lambda_\kp\,,
\end{equation}
where $g_k\coloneqq \norm[]{F^\prime(x_k)}$.
\end{lemma}
\begin{proof}
    The only difference from the proof of Lemma~\ref{lemma2:maindescent} lies in the estimation of $V(G_\kp)-V(\tilde G_\kp)$.
     However, we have to discuss by cases due to the restart step.
    \begin{enumerate}[1.]
        \item When $\trace\hat G_\kp\leq n\bar\kappa$, we have $\tilde G_\kp=\hat G_\kp$, $\trace G_\kp\leq n\bar\kappa$ and
        \begin{equation}
        V(G_\kp)-V(\tilde G_\kp) = \trace(G_\kp-G_\kp-\lambda_\kp \opid) \geq - n\lambda_\kp\,.
    \end{equation}
    \item When $\trace\hat G_\kp> n\bar\kappa$, we have $V(\hat G_\kp)>n\bar\kappa\geq nL$. Since in this case we set $\tilde G_\kp=L\opid$, we have $ V(\hat G_\kp)\geq V(\tilde G_\kp)$. We deduce that
    \begin{equation}
    \begin{split}
        V(G_\kp)-V(\tilde G_\kp)&= V(G_\kp)-V(\hat G_\kp) +V(\hat G_\kp)-V(\tilde G_\kp)\\&\geq V(G_\kp)-V(\hat G_\kp)  \\
       &\geq \trace(G_\kp-G_\kp-\lambda_\kp \opid)\\
       &\geq -n\lambda_\kp\\
        &\geq - n\lambda_\kp\,.\\
    \end{split}
    \end{equation}
    \end{enumerate}
    Therefore, no matter which case, we always have 
    \begin{equation}\label{ineq3:V1}
        V(G_\kp)-V(\tilde G_\kp)\geq - n\lambda_\kp\,.
    \end{equation}
    Summing \eqref{ineq3:V1} and \eqref{ineq2:V3},  we obtain the desired inequality. 
\end{proof}
Now, we are ready to prove our main theorem for Algorithm~\ref{Alg:SR1_da_PQN3}.
\begin{proof}[\textbf{Proof of Theorem~\ref{thm:main-gradreg}}]
    Following the same steps of the proof of Theorem~\ref{thm:main-grad} and using Lemma~\ref{lemma3:maindescent}, we obtain the global super-linear convergence rate  of Algorithm~\ref{Alg:SR1_da_PQN3}.
\end{proof}
% \begin{theorem}[Main Theorem 3]
% For any initialization $x_0\in\R^n$ and any $N\in\N$, grad reg SR1 PQN has a global convergence with the rate:
% \begin{equation}\label{ineq3:super-linearrate}
%     \norm[]{F^\prime (x_N)}\leq \left(\frac{C_{\mathrm{grad}}}{N^{1/4}}\right)^{N/2}\norm[]{F^\prime (x_0)} \,,
% \end{equation}
% where $C_{\mathrm{grad}}\coloneqq (\Theta + V(G_0))/\mu$, $C_0\coloneqq \sqrt{\frac{ 2(F(x_0) -\inf F)}{\mu}}$ and $\Theta = L_HC_0 +\sqrt{L_H(L+n\bar\kappa)C_0}$.
% \end{theorem}

\bibliography{main}
\bibliographystyle{siam}
\end{document}